\documentclass[12pt,leqno]{amsart}
\usepackage{a4wide}
\usepackage{amssymb}
\usepackage{amsmath}
\usepackage{amsthm}
\usepackage{verbatim}
\usepackage{graphicx}
\usepackage{psfrag}
\usepackage{xypic}
\numberwithin{equation}{section}

\theoremstyle{plain}
\newtheorem{theo}{Theorem}[section]
\newtheorem{lem}[theo]{Lemma}
\newtheorem{prop}[theo]{Proposition}
\newtheorem{cor}[theo]{Corollary}

\newtheorem{conj}[theo]{Conjecture}
\theoremstyle{remark}
\newtheorem{rem}[theo]{Remark}

\theoremstyle{definition}

\newcommand{\bF}{\bar{\mathbb F}}
\newcommand{\vol}{\hbox{vol}}
\newcommand{\ff}{\hbox{if }}
\newcommand{\Sym}{\hbox{Sym}}
\newcommand{\Aut}{\hbox{Aut}}

\renewcommand{\Im}{\operatorname{Im}}
\newcommand{\A}{{\mathbb A}}
\newcommand{\C}{{\mathbb C}}

\newcommand{\Q}{{\mathbb Q}}
\newcommand{\Z}{{\mathbb Z}}

\newcommand{\kzxz}[4]{\left(\begin{smallmatrix} #1 & #2 \\ #3 & #4\end{smallmatrix}\right) }
\newcommand{\kabcd}{\kzxz{a}{b}{c}{d}}

\newcommand{\Gal}{\operatorname{Gal}}

\newcommand{\CL}{\mathcal C\mathcal L}

\renewcommand{\div}{\operatorname{div}}

\newcommand{\Pet}{\text{\rm Pet}}
\newcommand{\CM}{\mathcal{CM}}

\newcommand{\norm}{\operatorname{N}}
\newcommand{\End}{\operatorname{End}}
\newcommand{\Hom}{\operatorname{Hom}}

\newcommand{\disc}{\operatorname{disc}}

\renewcommand{\O}{\mathcal O}
\newcommand{\GL}{\operatorname{GL}}

\newcommand{\tr}{\operatorname{tr}}
\newcommand{\ord}{\operatorname{ord}}
\newcommand{\SL}{\operatorname{SL}}

\newcommand{\cha}{{\text{\rm char}}}
\newcommand{\diag}{{\text{\rm diag}}}
\newcommand{\Sp}{\operatorname{Sp}}
\newcommand{\Cm}{\operatorname{CM}}
\newcommand{\CalM}{\mathcal M}
\newcommand{\CalZ}{\mathcal Z}
\newcommand{\Fal}{\operatorname{Fal}}
\newcommand{\CH}{\operatorname{CH}}

\begin{document}

\title[ Arithmetic Intersection and the Faltings  Height]{ Arithmetic Intersection on a Hilbert Modular Surface and the Faltings  Height }
\date{October 20, 2007, latest revision: January 29, 2009}
\author[ Tonghai Yang]{Tonghai Yang}
\address{Department of Mathematics, University of Wisconsin Madison, Van Vleck Hall, Madison, WI 53706, USA}
\email{thyang@math.wisc.edu} 
\thanks{partially supported  by grants DMS-0302043,DMS-0354353,  DMS-0555503,  NSFC-10628103, and Vilas Life Cycle Professorship (Univ. Wisconsin) }
\maketitle
\begin{abstract} In this paper, we prove an explicit arithmetic
intersection formula between arithmetic Hirzebruch-Zagier divisors
and arithmetic CM cycles in a Hilbert modular surface over $\mathbb
Z$. As applications, we obtain the first `non-abelian'
Chowla-Selberg formula, which is a special case of Colmez's
conjecture; an explicit arithmetic intersection formula between
arithmetic Humbert surfaces and CM cycles in the arithmetic Siegel
modular variety of genus two; Lauter's conjecture about the
denominators of CM values of Igusa invariants; and a result about
bad reductions of CM genus two curves.
\end{abstract}

\section{Introduction} \label{sect1}

 Intersection theory  has  played  a central  role not only in
 algebraic geometry 
 but also  in number theory and arithmetic geometry, such as Arakelov theory, Faltings's proof of Mordell conjecture,
 the Birch and Swinnerton-Dyer conjecture, and the Gross-Zagier formula, to name a few. In a lot of
 cases, explicit intersection formulae are needed  as in the  Gross-Zagier formula (\cite{GZ}),  its generalization
 to totally real number fields by Shou-Wu Zhang (\cite{Zh1}, \cite{Zh2}, \cite{Zh3}), recent work on arithmetic
 Siegel-Weil formula by Kudla, Rapoport, and the author
 (e.g., \cite{KuAnnal}, \cite{KR1}, \cite{KR2},  \cite{KRY1}, \cite{KRY2}), and Bruinier, Burgos-Gil, and K\"uhn's work
 on arithmetic Hilbert modular surfaces. In other cases, the explicit formulae are simply beautiful as in the work of
 Gross and Zagier on singular moduli \cite{GZ2},  the work of Gross and Keating on modular polynomials
 \cite{GK}(not to mention the really classical B\'ezout's theorem). In all
these works, intersecting cycles are of the same
 type and symmetric.

 In this paper, we consider the arithmetic intersection  of two natural families
 of cycles of {\it different type} in a Hilbert modular surface over $\mathbb Z$, arithmetic Hirzebruch-Zagier divisors
  and arithmetic CM cycles associated to  non-biquadratic quartic CM fields. They  intersect properly
  and have a conjectured  arithmetic intersection formula
  \cite{BY}.
  The main purpose of this paper is to prove the conjectured formula
  under a minor technical condition on the CM number field. As an
  application, we prove the first {\it non-abelian} Chowla-Selberg formula
  \cite{Col}, which is also a special case of Colmez's conjecture on
  the Faltings  height of CM abelian varieties. As another application,
  we obtain an explicit intersection formula between (arithmetic)
  Humbert surfaces and CM cycles in the (arithmetic) Siegel modular
  $ 3$-fold, which has itself two applications: confirming Lauter's
  conjecture on the
  denominators of Igusa invariants valued at  CM points \cite{La},
  \cite{YaICCM3}, and bad reduction of CM  genus two  curves.
 We also use the formula to  verify a variant of a conjecture
of Kudla on arithmetic Siegel-Weil formula.
   We  now set up notation and describe this work in a little more detail.

Let $D\equiv 1 \mod 4$ be a prime number, and let $F=\mathbb
Q(\sqrt D)$ with the ring of integers $\O_F=\mathbb
Z[\frac{D+\sqrt D}2]$ and different $\partial_F =\sqrt D \O_F$.
Let  $\mathcal M$ is the Hilbert moduli stack of assigning to a
base scheme  $S$ over $\mathbb Z$   the set of the triples $(A,
\iota, \lambda)$, where (\cite[Chapter 3]{Go} and \cite[Section
3]{Vo})

(1)\quad  $A$ is a abelian surface over $S$.

(2) \quad $\iota: \O_F \hookrightarrow \End_S(A)$ is real
multiplication of $\O_F$ on $A$.

(3) \quad $\lambda: \partial_F^{-1} \rightarrow P(A)=\Hom_{\O_F}(A,
A^\vee)^{\hbox{sym}}$ is a $\partial_F^{-1}$-polarization (in the
sense of Deligne-Papas) satisfying the condition:
$$
 \partial_F^{-1}\otimes A  \rightarrow A^\vee, \quad  r \otimes a
 \mapsto \lambda(r)(a)
 $$
 is an isomorphism.

 Next, for an integer $m \ge 1$, let $\mathcal T_m$ be the integral Hirzebruch-Zagier divisors in $\mathcal M$
 defined in \cite[Section 5]{BBK}, which is flat
 closure of the classical Hirzebruch-Zagier divisor $T_m$ in $\mathcal
 M$. We refer to Section \ref{sect2} for the modular interpretation of
 $\mathcal T_q$ when $q$ is split in $F$.

Finally,  let $K=F(\sqrt\Delta)$ be a quartic non-biquadratic CM
number field with real quadratic subfield $F$.  Let $\CM(K)$ be
the moduli stack over $\mathbb Z$ representing the moduli problem
which assigns to a base scheme $S$  the set of  the triples $(A,
\iota, \lambda)$ where $\iota: \O_K \hookrightarrow \End_S(A)$ is
an CM action of $\O_K$ on $A$, and $ (A, \iota|_{\O_F},
\lambda)\in \mathcal M(S)$ such that the Rosati involution
associated to $\lambda$ induces to the complex conjugation on
$\O_K$. The map $(A, \iota, \lambda) \mapsto (A, \iota|_{\O_F},
\lambda)$ is a finite proper map  from $\CM(K)$ into $\mathcal M$,
and we denote its direct image in $\mathcal M$ still by $\CM(K)$
by abuse of notation. Since $K$ is non-biquadratic, $\mathcal T_m$
and $\CM(K)$ intersect properly. A basic question is to compute
their arithmetic intersection number (see Section \ref{sect2} for
definition). Let $\Phi$ be a CM type of $K$ and let $\tilde K$ be
reflex field of $(K, \Phi)$. It is also a quartic non-biquadratic
CM field with real quadratic field $\tilde F= \mathbb
Q(\sqrt{\tilde D})$ with $\tilde D =\Delta \Delta'$. Here
$\Delta'$ is the Galois conjugate of $\Delta$ in $F$.

\begin{conj} (Bruinier and Yang) \label{conj}  Let the notation be as above and let $\tilde D=d_{\tilde F}$ be the discriminant of $\tilde F$.
Then
\begin{equation}
\mathcal T_m .\CM(K) = \frac{ 1}{2} b_m
\end{equation}
or equivalently
\begin{equation}\label{conjl}
(\mathcal T_m .\CM(K))_p = \frac{1}{2} b_m(p)
\end{equation}
for every prime $p$. Here $$ b_m =\sum_p b_m(p) \log p$$ is defined
as follows:
\begin{equation} \label{eqI1.3}
b_m(p) \log p =\sum_{\mathfrak p|p} \sum_{\substack{t=\frac{n+m\sqrt{\tilde
D}}{2D} \in d_{\tilde K/\tilde F}^{-1}\\ |n| <m \sqrt{\tilde D}}}
B_t(\mathfrak p)
\end{equation}
where \begin{equation} B_t(\mathfrak p)
 =\begin{cases}
  0 &\ff \mathfrak p \hbox{ is split in} \tilde K,
  \\
  (\ord_{\mathfrak p} t +1)
\rho(t d_{\tilde K/\tilde F} \mathfrak p^{-1}) \log |\mathfrak p|
 &\ff \mathfrak p \hbox{ is not split in} \tilde K,
\end{cases}
\end{equation}
$|\mathfrak p|$ is the norm of the ideal $\mathfrak p$ of $\tilde
F$, and
$$
\rho(\mathfrak a) =\#\{ \mathfrak A \subset \O_{\tilde K} :
N_{\tilde K/\tilde F} \mathfrak A =\mathfrak a\}.
$$
\end{conj}

Notice that the conjecture implies that $(\mathcal T_m.\CM(K))_p =0$
unless $4Dp\mid m^2 \tilde D -n^2$ for some integer $0 \le  n  <
m\sqrt {\tilde D}$. In particular, $\mathcal T_m.\CM(K)=0$ if $m^2
\tilde D \le 4 D$.

Throughout  this paper, we assume that $K$ satisfies the
following condition
\begin{equation} \label{eqOK}
\O_K =\O_F + \O_F \frac{w+\sqrt\Delta}2
\end{equation}
is free over $\O_F$($w \in \O_F$). 
The main result
of this paper is  the following theorem.

\begin{theo} \label{maintheo} Assume $(\ref{eqOK})$ and that
$\tilde D =\Delta\Delta' \equiv 1 \mod 4$ is a prime. Then
Conjecture \ref{conj} holds.
\end{theo}

The special case $m=1$ is proved in \cite{m=1}. Now we describe its
application to the generalized Chowla-Selberg formula. In
 proving the famous Mordell conjecture, Faltings introduces the
 so-called Faltings height $h_{\hbox{Fal}}(A)$ of an Abelian variety
 $A$, measuring the complexity of $A$ as a point in  a Siegel modular
 variety. When $A$ has complex multiplication, it  only depends on
 the CM type of $A$ and  has a simple description as follows. Assume that $A$ is defined over
 a number field $L$ with good reduction everywhere, and
 let $\omega_A \in \Lambda^g \Omega_A$ be a N\'eron differential
 of $A$ over $\O_L$, non-vanishing everywhere, Then
 the Faltings height of $A$ is defined as (our normalization is
 slightly different from that of \cite{Col})
 \begin{equation}
 h_{\hbox{Fal}}(A) =-\frac{1}{2[L:\Q]}
    \sum_{\sigma: L \hookrightarrow \mathbb C} \log
     \left|(\frac{1}{2 \pi i})^g \int_{ \sigma(A)(\C)} \sigma(\omega_A) \wedge
     \overline{\sigma(\omega_A)}\right| + \log \# \Lambda^g \Omega_A/\O_L\omega_A.
     \end{equation}
Here $g=\dim A$.  Colmez gives a beautiful conjectural formula to
 compute the Faltings height of a CM abelian variety
 in terms of the log derivative of certain Artin L-series associated to the CM type  \cite{Col},
 which is consequence of   his  product formula conjecture of $p$-adic
 periods in the same paper. When $A$ is a CM elliptic curve, the
 height conjecture is a reformulation of  the well-known Chowla-Selberg
 formula relating the CM values of the usual  Delta function
 $\Delta$ with the values  of  the  Gamma function  at rational
numbers. Colmez proved his conjecture  up to a multiple of $\log 2$
when the CM field (which acts on $A$) is abelian, refining Gross's
 \cite{Gr} and Anderson's  \cite{An} work. A key point is that such CM abelian varieties
 are isogenous
quotients of the Jacobians of the Fermat curves, so one has a model
to work with.  K\"ohler and Roessler gave a different proof of a
weaker version of Colmez's result using their Lefschetz fixed point
theorem in Arakelov geometry \cite{KRo} without using explicit model
of CM abelian varieties. They still relied on the action of $\mu_n$
on product of copies of these CM abelian varieties, and did  not
thus  break the barrier of non-abelian CM number fields. V. Maillot
and Roessler gave a more general conjecture relating logarithmtic
derivative or (virtual) Artin L-function with motives and provided
some evidence in \cite{MR} (weaker than  the Colmez conjecture when
restricting to CM abelian varieties) and Yoshida independently
developed a conjecture about absolute CM period which is very close
to Colmez's conjecture and provided some non-trivial numerical
evidence as well as partial results \cite{Yo}. We should also
mention that Kontsevich and Zagier \cite{KZ} put these conjectures
in different perspective in the framework of periods, and for
example rephrased the Colmez conjecture (weaker form) as saying the
log derivative of Artin L-functions is a period.

When  the CM number field is {\it non-abelian},  nothing is known
about Colmez's conjecture. In this paper we consider the case that
$K$ is a non-biquadratic quartic CM number field (with real
quadratic subfield $F$), in which case Colmez's conjecture can be
stated precisely as follows.  Let $\chi$ be the quadratic Hecke
character of $F$ associated to $K/F$ by the global class field
theory, and let
\begin{equation}
\Lambda(s, \chi) = C(\chi)^{\frac{s}2}
\pi^{-s-1}\Gamma(\frac{s+1}2)^2 L(s, \chi)
\end{equation}
be the complete L-function of $\chi$ with $C(\chi) =D N_{F/\mathbb
Q} d_{K/F}$. Let
\begin{equation}
\beta(K/F)
 = \frac{\Gamma'(1)}{\Gamma(1)}
-\frac{\Lambda'(0, \chi )}{\Lambda(0, \chi )} -\log4\pi .
\end{equation}
In this case, the conjectured formula of Colmez on the Faltings
height of a CM abelian variety $A$ of type $(K, \Phi)$ does not even
depend on the CM type $\Phi$ and is given by (see \cite{Ya3})
\begin{equation} \label{eqCol}
h_{\hbox{Fal}}(A)=\frac{1}2 \beta( K/ F) .
\end{equation}

In Section \ref{sect7}, we will prove the following result  using
Theorem \ref{maintheo}, and \cite[Theorem 1.4]{BY}, which breaks the
barrier of `non-abelian' CM number fields. Our proof is totally
different.

\begin{theo} \label{Colmez} Assume that $K$ satisfies the conditions in \ref{maintheo}.  Then Colmez's conjecture
(\ref{eqCol}) holds.
\end{theo}

  Kudla initiated a program to relate the arithmetic intersections on
 Shimura varieties over $\mathbb Z$ with the derivatives of
 Eisenstein series---{\it arithmetic Siegel-Weil Formula} in 1990's,
 see \cite{KuAnnal}, \cite{KuMSRI}, \cite{KRY2}  and references there for example. Roughly speaking, let
\begin{equation}\label{eq1.9}
\hat\phi(\tau) =-\frac{1}2 \hat{\omega} + \sum_{m>0} \hat{\mathcal
T}_m q^m
\end{equation}
be the modular form of weight 2, level $D$, and character
$(\frac{D}{})$ with values in the arithmetic Chow group defined by
Bruinier, Burgos Gil, and K\"uhn \cite{BBK} (see also Section
\ref{sect7}), where $\hat{\omega}$ is the metrized Hodge bundle on
$\tilde{M}$ with Peterson metric defined in Section \ref{sect7} and
can be viewed as an arithmetic Chow cycle, and $\hat{\mathcal T}_m$
is some arithmetic Chow cycle related to $\mathcal T_m$. Then we
have the following result, which can be viewed as a variant of
Kudla's conjecture in this case. We refer to Theorem \ref{theo7.1}
for more precise statement of the result.

\begin{theo} \label{theo1.4}  Let the notation and assumption be as in Theorem
\ref{maintheo}. Then $ h_{\hat{\phi}}(\CM(K))
 +\frac{1}4 \Lambda(0, \chi) \beta(K/F) E_2^+(\tau)
 $
 is the holomorphic projection of  the diagonal restriction of  the central derivative of some (incoherent)
 Hilbert Eisenstein series on $\tilde F$. Here $E_2^+(\tau)$ is an Eisenstein
 series of weight $2$.
 \end{theo}

Let $\mathcal A_2$ be the moduli stack of principally polarized
abelian surfaces \cite{CF}. $\mathcal A_2(\mathbb C) = \Sp_2(\mathbb
Z) \backslash \mathbb H_2$ is the Siegel modular variety of genus
$2$. For each integer $m$, let $G_m$ be the Humbert surface in
$\mathcal A_2(\mathbb C)$ (\cite[Chapter 9]{Ge}, see also Section
\ref{sect8}), which is actually defined over $\mathbb Q$. Let
$\mathcal G_m$ be the flat closure of $G_m$ in  $\mathcal A_2$.  For
a quartic  CM number field $K$, let $\CM_S(K)$ be the moduli stack
of principally polarized CM abelian surfaces by $\O_K$. In Section
8, we will prove the following theorem  using Theorem \ref{maintheo}
and a natural  map from $\mathcal M$ to $\mathcal A_2$.

\begin{theo} \label{siegel} Assume $K$ satisfies the condition in Theorem
\ref{maintheo}, and that $Dm$ is not a square. Then $\CM_S(K)$ and
$\mathcal G_m$ intersect properly, and
\begin{equation}
\CM_S(K).\mathcal G_m = \frac{1}2\sum_{n >0,  \frac{D m-n^2}{4}
\in \mathbb Z_{>0}} b_{\frac{Dm-n^2}4}.
\end{equation}
\end{theo}

Since  $\mathcal G_1$ is the moduli space of
principally polarized abelian  surfaces which are not Jacobians of
genus two curves,  the above theorem has the following
consequence.

\begin{cor} \label{genus2} Let $K$ be a  quartic CM number field as in Theorem
 \ref{maintheo}. Let $C$ be a genus two curve over a number field
 $L$ such that its Jacobian $J(C)$ has CM by $\O_K$ and has good
 reduction everywhere.  Let $l$ be a prime. If   $C$ has bad reduction at a
 prime $\mathfrak l|l$ of $L$, then
 \begin{equation} \label{eq8.4}
 \sum_{0 < n <\sqrt D, odd} b_{\frac{D-n^2}4}(l) \ne 0
\end{equation}
 In particular, $l\le \frac{D \tilde D}{64}$. Conversely, if
 $(\ref{eq8.4})$ holds for a prime $l$, then there is a genus two curve $C$ over
 some
 number field $L$ such that

 (1) \quad  $J(C)$ has CM by $\O_K$ and has good reduction
 everywhere, and

 (2) \quad $C$ has bad reduction at a prime $\mathfrak l$ above $l$.
 \end{cor}

 Finally we recall  that  Igusa
defines 10 invariants which characterize genus two curves over
$\mathbb Z$ in  \cite{Ig2}. They are Siegel modular forms of genus
$2$ (level $1$) \cite{Ig}. One needs three of them to determine
genus two curves over $\bar{\mathbb Q}$ and over $\bar{\mathbb F}_p$
for $p
>5$, which are now  called the Igusa invariants $j_1$, $j_2$, and
$j_ 3$. Recently,  Cohn and Lauter (\cite{CL}), and Weng \cite{Wen}
among others started to use genus two curves over finite fields for
cryptosystems. For this purpose, they need to compute the CM values
of the Igusa invariants associated to a quartic non-biquadratic CM
field. Similar to the classical $j$-invariant, these CM values are
algebraic numbers. However, they are in general  not algebraic
integers. It is very desirable to at least bound the denominators of
these numbers for this purpose and also in theory. Lauter gives an
inspiring conjecture about the denominator  in \cite{La} based on
her calculation and Gross and Zagier's work on singular moduli
\cite{GZ}. In Section \ref{sect8}, we will  prove the following refinement of
her conjecture subject to the condition in Theorem \ref{maintheo}.

\begin{theo} \label{Lauter} (Lauter's conjecture). Let $j_i'$, $i=1, 2, 3$ be the
slightly renormalized Igusa invariants in Section \ref{sect8}, and
let $\tau$ be a CM point in $X_2$ such that the associated abelian
surface $A_\tau$ has endomorphism ring $\O_K$, and let $H_i(x)$ be
the minimal polynomial of $j_i'(\tau)$ over $\mathbb Q$. Assume
$K$ satisfies the condition in Theorem \ref{maintheo}. Let $A_i$ be positive integers given by
$$
A_i =\begin{cases}
  e^{3 W_K \sum_{0 < n <
\sqrt D, odd} b_{\frac{D-n^2}4}} &\ff  i=1,
\\
 e^{2 W_K \sum_{0 < n <
\sqrt D, odd} b_{\frac{D-n^2}4}} &\ff  i=2, 3. \end{cases} $$ Here
$W_K$ is the number of roots of unity in $K$. Then $A_i H_i(x)$ is
defined over $\mathbb Z$. In particular, $A_i \norm( j_i'(\tau))$
is a rational integer.
\end{theo}

Now we describe briefly how to prove Theorem \ref{maintheo} and its
consequences. The major effort is to prove the following weaker
version of the main theorem, which covers Sections
\ref{sect2}-\ref{sect6}.

\begin{theo} \label{theo1.5} Assume
(\ref{eqOK}) and that $\tilde D =\Delta \Delta' \equiv 1 \mod 4$ is
square free, and that $q$ is an odd prime split in $F$. Then
\begin{equation}
\mathcal T_q.\CM(K) =\frac{1}2 b_q + c_q \log q
\end{equation}
for some rational number $c_q$. Equivalently, one has for any
prime $p \ne q$,
\begin{equation}
(\mathcal T_q.\CM(K))_p = \frac{1}2 b_q(p).
\end{equation}
\end{theo}

  The starting point is a proper map
from the moduli stack $\mathcal Y_0(q)$ of cyclic isogeny $(\phi:
E \rightarrow E')$ of degree $q$ of elliptic curves to $\mathcal
T_q$ constructed by Bruinier, Burgos-Gil, and K\"uhn in
\cite{BBK}, see also Section \ref{sect2}. Let $(B, \iota,
\lambda)$ be the image of $(\phi: E \rightarrow E')$ in $\mathcal
T_q$, we first compute the endomorphism ring of $(B, \iota)$ in
terms of a pair of quasi-endomorphisms $\alpha, \beta \in
\phi^{-1}\Hom(E, E')$ satisfying some local condition at $q$.
This is quite different from the special case $q=1$ considered in \cite{m=1}:
we can not describe the endomorphism ring of $(E, \iota)$ globally.
The upshot  is the following:
associated to a geometric intersection point in $\mathcal
T_q.\CM(K)(\bar{\mathbb F}_p)$ is a triple $(\phi, \phi\alpha,
\phi\beta: E\rightarrow E')$ satisfying certain {\it local}
condition at $q$. Using a beautiful formula of Gross and Keating
\cite{GK} on deformation of isogenies, we are able to compute the
local intersection index and prove the following theorem.

\begin{theo} (Theorem 3.6) For $p \ne q$, one has
$$
(\mathcal T_q.\CM(K))_p = \frac{1}4 \sum_{\substack{ 0 < n < q
\sqrt{\tilde D} \\ \frac{q^2 \tilde D-n^2}{4D} \in p\mathbb Z_{>0}}}
\left( \ord_p \frac{q^2 \tilde D -n^2}{4D} +1\right) \sum_\mu
\sum_{[\phi: E \rightarrow E'] } \frac{R(\phi, T_q(\mu
n))}{\#\hbox{Aut}(\phi)} .
$$
Here $\mu =\pm 1$, $T_q(\mu n)$ is a positive definite $2\times 2$
matrix with entries in $\frac{1}q \mathbb Z$ determined by $n$ and
$\mu$ as in Lemma \ref{lemold1.1}.
 $R(\phi, T_q(\mu n))$ is the number of pairs $(\delta, \beta)
\in (\phi^{-1}\Hom(E, E'))^2$ satisfying certain local condition at
$q$ and $2$ such that
$$T(\delta, \beta):=\frac{1}2 \kzxz {(\delta, \delta)} {(\delta,
\beta)} {(\delta, \beta)} {(\beta, \beta)}   =T_q(\mu n).$$ Finally,
$\Aut(\phi)$ is the set of automorphisms $f \in \Aut(E)$ such that
$\phi \circ f \circ \phi^{-1} \in \Aut(E')$, and the summation is
over the equivalence classes of all isogenies  $[\phi: E \rightarrow
E']$ of degree $q$ of supersingular elliptic curves over
$\bar{\mathbb F}_p$.
\end{theo}
 The next step is to compute the summation
 $$
 \beta(p, \mu n) = \sum_{[\phi: E\rightarrow E']}\frac{R(\phi, T_q(\mu
n))}{\#\hbox{Aut}(\phi)}
$$
which counts the `number' of geometric intersection points between
$\CM(K)$ and $\mathcal T_q$ at $p$. The sum  can be written as
product of local Whittaker integrals and can be viewed as a
generalization of quadratic local density.  In theory, the idea in
\cite{YaDensity1}, \cite{YaDensity2} can be generalized to compute
these local integrals, but it is very complicated. In Section
\ref{sect4}, we take advantage of the relation between supersingular
elliptic curves and maximal orders of the quaternion algebra
$\mathbb B$ which ramifies only at $p$ and $\infty$, and known
structure of quaternions, and transfer the summation into product of
local integral over $\mathbb B_l^*$ instead of usual local density
integral as in \cite{YaDensity1}, \cite{YaDensity2}:

\begin{equation}
\beta(p, \mu n) =\frac{1}2  \int_{\mathbb Q_f^* \backslash \mathbb
B_f^* /\mathcal K} \Psi(g^{-1}.\vec x_0) dg
\end{equation}
if there is $\vec x_0=V(\A_f)^2$ with $T(\vec x_0) = T_q(\mu n)$.
Otherwise, $\beta(p, \mu n) =0$. Here
$$g.\vec x= (g.X_1, g.X_2) =(gX_1g^{-1}, gX_2g^{-1}), \quad \vec
x={}^t(X_1, X_2),
$$
and $\Psi=\prod \Psi_l \in S(V(\A_f))^2$  and $V$ is the quadratic
space of trace zero elements in $\mathbb B$. In Section \ref{sect5}, we
compute these local integrals which is quite technical at $q$ due to
the local condition mentioned above, and obtain an explicit formula
for $\beta(p, \mu n)$ (Theorems \ref{theo5.1} and \ref{theo5.2}). In
Section \ref{sect6}, we compute $b_m(p)$ and proves Theorem \ref{theo1.5}. The
computation also gives a more explicit formula for the intersection
number.

 In Section \ref{sect7}, we use the height pairing  function
 and \cite[Theorem 1.4]{BY} to derive the main theorem from
the weaker version.  we also derive Theorem \ref{Colmez} from
Theorem \ref{maintheo} using  the same idea. Theorem \ref{theo1.4}
is a consequence of the main theorem and \cite[Theorem 8.1]{BY}.  In
Section \ref{sect8}, we briefly review the natural modular
`embedding' from Hilbert modular surfaces to the Siegel modular
$3$-fold,  and prove Theorems \ref{siegel}, \ref{genus2}, and
\ref{Lauter}.

\textbf{Acknowledgments.} To be added.

\section{A brief review of the case $q=1$}
\label{newsect2}

For the convenience of the reader, we briefly review the
computation of the arithmetic intersection between $\CM(K)$ and
$\mathcal T_q$ in the very special case $q=1$ to give a  rough
idea and motivation  to the general case considered in this paper.
We also briefly describe how  Gross and Zagier's beautiful
factorization formula for singular moduli can be derived this way.
We refer to \cite{m=1} for detail, and to Section \ref{sect2} for
notation.

 Let $\mathcal E$ be the moduli stack over $\mathbb Z$ of elliptic
curves. Then  there is a natural isomorphism between $\mathcal E$
and $\mathcal T_1$ given by  $E \mapsto (E\otimes \O_F, \iota,
\lambda)$.  A simple but critical fact is that
$\End_{\O_F}(E\otimes \O_F) \cong \End(E) \otimes \O_F$ is easy to
understand (it is much more complicated even in the split prime
$q$ case considered in Section \ref{sect2}). So a geometric
intersection point in $\mathcal T_1.\CM(K) (\bar{\mathbb F}_p)$ is
determined by a pair $(E, \iota)$ where
$$
\iota: \O_K \hookrightarrow \End(E) \otimes \O_F
$$
such that the main involution on $\O_E=\End(E)$ gives the complex
conjugation on $\O_K$, which implies in particular that $E$ is
supersingular and $p$ is inert in $F$. Since we assume that
$\O_K=\O_F + \O_F \frac{w+\sqrt\Delta}2$, $\iota$ is determined by
$$
\iota(\frac{w+\sqrt\Delta}2) = \alpha_0 + \beta_0 \frac{D +\sqrt
D}2, \quad \iota(\sqrt\Delta) = \delta + \beta \frac{D +\sqrt D}2,
$$
with $\alpha_0, \beta_0 \in \O_E$, and
$$ \delta=2\alpha_0 -w_0,
\beta= 2\beta_0-w_1 \in L_E=\{ x \in \mathbb Z + 2 \O_E:\, \tr x=0\}.
 $$
Here $w=w_0 + w_1 \frac{D+\sqrt D}2$ with $w_i \in \mathbb Z$. Set
$$
T(\delta,\beta) = \frac{1}2 \kzxz {(\delta, \delta)} {(\delta,
\beta)}  {(\delta, \beta)} {(\beta, \beta)} \in \Sym_2(\mathbb Z).
$$
One shows that $T(\delta, \beta)$ is a positive definite integral
matrix of the form  $T_1(\mu n)$ (in the notation of Lemma
\ref{lemold1.1})  for a unique positive integer $n$ with $\det
T_1(\mu n) =\frac{\tilde D-n^2}{D} \in 4p\mathbb Z_{>0}$ and a
unique sign $\mu =\pm 1$.

  Applying a beautiful deformation result of Gross and Keating to
  $1$, $\alpha_0$, and $\beta_0$, we show in \cite[Section 4]{m=1} that
  the local intersection index  of  $\mathcal T_1$ and $\CM(K)$ at
  $ (E, \iota)$ is given by
  $$
  \iota_p (E,  \iota)
   = \frac{1}2(\ord_p \frac{\tilde D-n^2}{4D} +1)
     $$
    which  depends only on $n$. So the  intersection number of
     $\mathcal T_1$ and $\CM(K)$ at $p$ is
   $$
   (\mathcal T_1.\CM(K))_p = \frac{1}2 \sum_{\frac{\tilde D
   -n^2}{4D} \in p\mathbb Z_{>0}} (\ord_p \frac{\tilde D-n^2}{4D} +1)
   \sum_{\mu} \sum_{E s.s.} \frac{R(L_E, T_1(\mu n))}{\# \Aut(E)}
   $$
   where the sum is running over all supersingular elliptic curves
   over $\bar{\mathbb F}_p$ (up to isomorphism), and $R(L_E,
   T_1(\mu n))$ is the representation number of the ternary
   quadratic form $L_E$ representing the matrix $T_1(\mu n)$.

  Finally the last sum is easily seen to be  the product of local
  densities, and can be computed using the formulae in
  \cite{YaDensity1} and \cite{YaDensity2}. However, the case $p=2$
  is extremely complicated, so we used a trick  in \cite{m=1} to
  switch it a local density problem of  $\O_E$ (with the reduced
  norm as its quadratic form) representing a symmetric $3\times 3$
  matrix related to $T_1(\mu n)$, which is computed in \cite{GK}.
  This trick only works in this special case since $\O_E$ is very
  special. In general local density of representing a $3\times 3$
  matrix by a quadratic form of higher dimension is extremely
  complicated. We will have to use a new idea to deal with the
  case $q \ne 1$ in Sections \ref{sect4} and \ref{sect5}. The
  upshot is  then the following formula:
  $$
   (\mathcal T_1.\CM(K))_p = \frac{1}2 \sum_{\frac{\tilde D
   -n^2}{4D} \in p\mathbb Z_{>0}} (\ord_p \frac{\tilde D-n^2}{4D} +1)
   \sum_{\mu} \beta(p, \nu n)
   $$
   where
   $$
   \beta(p, \mu n) =\prod_{l | \frac{\tilde D-n^2}{4 D}}
   \beta_l(p, \mu n)
   $$
   and $\beta_l(p, \mu n)$ is given by right hand side of the formula in Theorem \ref{theo5.1}.
   This finishes the computation at the geometric side. On the
   algebraic side,
    the computation of $b_1(p)$  is similar to that of $b_m(p)$
    in  Section \ref{sect6}(of course simpler) and shows that $b_1(p)$ is the equal to the right hand side of the above
    formula without the factor $\frac{1}2$.  That proves the
    case $q=1$.

     If we further allow $D=1$, i.e., $F=\mathbb Q \oplus Q$, and
     $K= \mathbb Q(\sqrt{d_1}) \oplus \mathbb Q(\sqrt{d_2})$, one
     has  $\mathcal M =\mathcal E  \times \mathcal E$  and $\CM(K)
     =\CM(d_1) \times \CM(d_2)$ where $\CM(d_i)$ is the moduli
     stack of CM elliptic curves of (fundamental) discriminant $d_i
     <0$. Furthermore, $\mathcal T_1$ is just the diagonal
     embedding of $\mathcal E$. From this, it is easy to see
     \begin{align}
\mathcal T_1.\CM(K)
 &=  \CM(K_1).\CM(K_2) \quad \hbox{ in } \quad \mathcal M_1 \notag
 \\
  &= \sum_{\disc[\tau_i]=d_i} \frac{4}{w_1 w_2} \log |j(\tau_1)
  -j(\tau_2)|
  \end{align}
  where $w_i=\# \O_i^*$ and $\tau_i$ are Heegner points in
  $\mathcal M_1(\mathbb C)$ of discriminant $d_i$. Now the
  beautiful factorization of Gross-Zagier on singular moduli
  follows from the arithmetic intersection formula for $\mathcal
  T_1.\CM(K)$. We refer to \cite[Section 3]{m=1} for detail.

\section{Modular Interpretation of $\mathcal T_q$ and Endomorphisms
of Abelian varieties} \label{sect2}

Let $F=\mathbb Q(\sqrt D)$ with $D \equiv 1 \mod 4$ prime. Let
$\mathcal M$ be the Hilbert modular stack defined in the
introduction, and let $\tilde{\mathcal M}$ be a fixed Toroidal
compactification. Let $K= F(\sqrt\Delta)$ be a non-biquadratic
quartic CM number field with real quadratic  subfield $F$, and let
$\CM(K)$ be the CM cycle defined  in the introduction. Notice that
$\CM(K)$ is closed in $\tilde{\CalM}$. $K$ has four different CM
types $\Phi_1$, $\Phi_2$, $\rho \Phi_1=\{ \rho\sigma: \, \sigma
\in \Phi_1\}$, and $\rho\Phi_2$, where $\rho$ is the complex
conjugation in $\mathbb C$. If $x=(A, \iota, \lambda) \in
\CM(K)(\mathbb C)$, then $(A, \iota, \lambda)$ is a CM abelian
surface over $\mathbb C$ of exactly one CM type $\Phi_i$ in
$\mathcal M(\mathbb C)=\SL_2(\O_F) \backslash \mathbb H^2$ as
defined in \cite[Section 3]{BY}. Let $\Cm(K, \Phi_i)$ be set of
(isomorphism classes) of CM abelian surfaces of CM type $(K,
\Phi_i)$ as in \cite{BY}, viewed as a cycle in $\mathcal M(\mathbb
C)$. Then it was proved in \cite{BY}
$$
\Cm(K) =\Cm(K, \Phi_1) + \Cm(K, \Phi_2) =  \Cm(K, \rho \Phi_1) +
\Cm(K, \rho \Phi_2)
$$
is defined over $\mathbb Q$. So we have

\begin{lem} \label{newlem2.1} One has
$$
\CM(K)(\mathbb C) = 2\Cm(K)
$$
in   $\CalM(\mathbb C)$.
\end{lem}

Next for an integer $m>0$, let $T_m$ be the Hirzebruch-Zagier
divisor $T_m$  is given by \cite{HZ}
$$
T_m(\mathbb C)=\SL_2(\O_F)\backslash \{ (z_1, z_2) \in \mathbb
H^2: (z_2,
1)A  \left(\substack{ z_1 \\
1}\right) =0 \hbox{ for  some } A \in L_m \},
$$
where
$$
L_m=\{ A =\kzxz {a} {\lambda} {\lambda'} {b} :\, a, b \in \mathbb
Z, \lambda \in
\partial_F^{-1}, ab -\lambda \lambda' =\frac{m}D\}.
$$
 $T_m$
is empty if $(\frac{D}m)=-1$. Otherwise, it is  a finite union of
irreducible curves and is actually defined over $\mathbb Q$.
Following \cite{BBK}, let $\mathcal T_m$ be the flat closure of
$T_m$ in $\mathcal M$, and let $\tilde{\mathcal T}_m$ be  the
closure of $\mathcal T_m$ in $\tilde{\mathcal M}$. When $m=q$ is a
prime split in $F$, $\mathcal T_m$ has the following modular
interpretation. Notice that our $\mathcal T_m$ might be different
from the arithmetic Hirzebruch-Zagier divisor $\mathcal T_m$ defined
in \cite{KR1} using moduli problem, although they are closely
related. It should be interesting to find out their precise
relation.

Let $q $ be a prime number split in $F$, and let $\mathfrak q$ be
a fixed prime of $F$ over $q$. In this paper, we will fix an
identification  $F \hookrightarrow F_{\mathfrak q} \cong \mathbb
Q_q$, and let $\sqrt D \in \mathbb Q_q$ be the image of $\sqrt
D\in F$ under the identification. Following \cite{BBK}, we write
$\mathfrak q =r \mathfrak c^2$ with some $r \in F^*$ of norm being
a power of $q$ and some fractional ideal $\mathfrak c$ of $F$. For
a cyclic isogeny $\phi: E \rightarrow E'$ of elliptic curves of
degree $q$ over a scheme $S$ over $\mathbb Z[\frac{1}q]$,
Bruinier, Burgos, and K\"uhn constructed a triple $(B, \iota,
\lambda)$ as follows. First let $A=E \otimes \mathfrak c$, and
$B=A/H$ with $H=(\ker \phi\otimes \mathfrak c)\cap A[\mathfrak
q]$.  We have the following commutative diagram:

\begin{align} \label{diag2.1}
\xymatrix{ A=E\otimes \mathfrak c \ar[rr]^-{\pi_\mathfrak q}
\ar[dd]^-{\phi\otimes 1} \ar[dr]^-{\pi} & &A/A[\mathfrak q] \cr
 &
B=A/H\ar[ur]^-{\pi_2} \ar[dl]^{\pi_1} & \cr  A'=E'\otimes\mathfrak c
 \cr
 }
\end{align}

The natural action of $\O_F$ on $A$ induces an  action $\iota: \O_F
\hookrightarrow \End(B)$. It is clear
\begin{equation}
P(A)=\Hom_{\O_F}(A, A^\vee)^{\Sym} =\mathfrak c^{-2} \partial_F^{-1}
\end{equation}
naturally.  They proved that under the natural injection
$$
P(B) \hookrightarrow P(A),\quad g \mapsto \pi^\vee g \pi
$$
the image of $P(B)$ is $\partial_F^{-1}$. This gives the
Deligne-Pappas $\partial^{-1}$-polarization map $$ \lambda:
\partial_F^{-1} \rightarrow P(B)$$ satisfying the Deligne-Papas
condition. Furthermore, they proved \cite[Proposition 5.12]{BBK}
that
\begin{equation} \label{eq2.3}
\xymatrix{\Phi: \, (\phi: E \rightarrow E') \mapsto  (B, \iota,
\lambda) \cr}
\end{equation}
is a proper map from the moduli stack $\mathcal Y_0(q)$ over
$\mathbb Z[\frac{1}q]$ to $\mathcal M$, and $\mathcal T_q = \Phi_*
\mathcal Y_0(q)$. The map $\Phi$ is  generically  an isomorphism.  This
proper map extends to a proper map from $\mathcal X_0(q)$ to
$\tilde{\mathcal M}$, whose direct image is the closure
$\tilde{\mathcal T}_m$ of $\mathcal T_m$ in $\tilde{\mathcal M}$.

Recall \cite{Gi}, \cite[Section 1]{Ho}, \cite[Chapter 2]{KRY2},
 \cite{Vi}, and \cite[Section 2]{m=1} that
 two cycles $\mathcal Z_i$ in a DM-stack $\mathcal X$ of
 codimension $p_i$, $p_1 +p_2 = \dim \mathcal X$, intersect
 properly if $\mathcal Z_1\cap\mathcal Z_1= \mathcal Z_1 \times_{\mathcal X} \mathcal Z_2$ is a
 DM-stack of dimension $0$. In such a case, we define their
 (arithmetic) intersection number as  \begin{equation}
 \mathcal Z_1.\mathcal Z_2=\sum_{p} \sum_{x \in \mathcal Z_1
 \cap \mathcal Z_2(\bar{\mathbb F}_p) }\frac{1}{\#\Aut(x)} \log \#
 \tilde{\O}_{\CalZ_1\cap \CalZ_2, x}
    =\sum_{p} \sum_{x \in \mathcal Z_1
 \cap \mathcal Z_2(\bar{\mathbb F}_p) } \frac{1}{\#\Aut(x)}
 i_p(\CalZ_1, \CalZ_2, x) \log p
 \end{equation}
 where $\tilde{\O}_{\CalZ_1\cap\CalZ_2, x}$ is the strictly local
 henselian ring of $\CalZ_1\cap\CalZ_2$ at $x$,
 $$
 i_p(\CalZ_1, \CalZ_2, x) = \operatorname{Length} \tilde{\O}_{\CalZ_1\cap\CalZ_2, x}
 $$
 is the local intersection index of $\CalZ_1$ and $\CalZ_2$ at $x$. If $\phi: \CalZ \rightarrow \CalM$ is a finite proper and
 flat map from stack $\CalZ$ to $\CalM$, we will identify  $\CalZ$
 with its direct image $\phi_*\CalZ$ as a cycle of $\CalM$, by abuse
 of notation.

  Now come back to our special case. Let $p \ne q$ be a fixed prime.  consider the
  diagram over $\mathbb Z_p$
\begin{align} \label{diag2.2}
\xymatrix{ \CM(K)\times_{\mathcal M} \mathcal Y_0(q)  \ar[r]
\ar[d] &\mathcal Y_0(q) \ar[d] \cr
 \CM(K) \ar[r] &  \mathcal M
 \cr
 }
\end{align}
One sees that a geometric point in $\CM(K) \cap \mathcal T_q$ is
indexed by a pair $x=( \phi: E \rightarrow E', \iota)$ with $\phi
\in \mathcal Y_0(\bar{\mathbb F}_p)$ and $\iota: \O_K
\hookrightarrow \End_{\O_F} (B)$ is an $\O_K$-action on $B$ such
that the Rosati involution associated to $\lambda$ gives the complex
conjugation on $K$.  Since $K$ is a quartic non-biquadratic CM
number field, one sees immediately that such a geometric point does
not exist unless $p$ is nonsplit in $F$ and $E$ is supersingular. In
such a case, write $I(\phi)$ for all $\O_K$ action $\iota$
satisfying the above condition. Then the intersection number of
$\CM(K)$ and $\mathcal T_q$ at $p$ is given by
\begin{equation} \label{intersection}
(\CM(K).\mathcal T_q)_p = \sum_{\phi \in \mathcal
Y_0(q)(\bar{\mathbb F}_p), \iota \in I(\phi) }
\frac{1}{\#\Aut(\phi)} i_p(\CM(K), \mathcal T_q, (\phi, \iota)) log
p. \end{equation}
Let $W$ be the Witt ring of $\bar{\mathbb F}_p$.
Let $\mathbb E$ and $\mathbb E'$ be the universal deformations  of
$E$ and $E'$ to $W[[t]]$ and $W[[t']]$ respectively. Let $I$ be the
minimal ideal of $W[[t, t']]$ such that

(1) \quad $\phi$ can be lifted to an (unique) isogeny $\phi_I:
 E_I  \rightarrow E'_I$, where $E_I =\mathbb E \mod I$ and $E_I' =
 \mathbb E' \mod I$.

(2) \quad Let $(B_I, \iota_I, \lambda_I) \in \mathcal M(W[[t,
t']]/I)$ be associated to $\phi_I$. The embedding $\iota$ can be
lifted to an embedding $\iota_I: \O_K \hookrightarrow
\End_{\O_F}(B_I)$.

By deformation theory, one can show that the local intersection
index is equal to
\begin{equation} \label{neweq2.6}
i_p(\phi, \iota):= i_p(\CM(K), \mathcal T_q,
(\phi, \iota)) = \operatorname{Length} W[[t, t']]/I.
\end{equation}

To compute the local intersection index and to count the geometric
intersection points. Let
$(\phi: E \rightarrow E') \in \mathcal Y_0(q)$ and let $(B, \iota,
\lambda) =\Phi(\phi) \in \mathcal M$. Then
$$
\End_{\O_F} B =\{g \in \End_S B:\, \iota(r) g =g \iota(r), r \in
\O_F\}.
$$
 We first make the following identification
\begin{equation} \label{oldeq2.2}
\pi^*:  \End_{\O_F}^0B=\End_{\O_F}B \otimes \mathbb Q \cong
\End^0(A) =\End^0(E) \otimes_\mathbb Z \O_F,  g \mapsto \pi^{-1}
\circ g \circ \pi = \frac{1}q  \pi^\vee g \pi.
\end{equation}

\begin{lem} Under the identification (\ref{oldeq2.2}), we have
$$
\End(\phi) \otimes \O_F \subset \pi^* \End_{\O_F}(B) \subset
\phi^{-1} \Hom(E, E') \otimes \O_F.
$$
Here
$$
\End(\phi) =\{ f \in \End(E):\,  \phi  f \phi^{-1} \in \End(E') \}.
$$
\end{lem}
\begin{proof} For $f \in \End(\phi)$,  and $x \in H$, let $f'=\phi f
\phi^{-1} \in \End(E')$, one has 
$$
(\phi \otimes 1)((f\otimes 1)(x)) = (f'\otimes 1)(\phi \otimes 1)(x)
=0
$$
and so $(f \otimes 1)( x) \in \ker (\phi \otimes 1) =\ker \phi
\otimes \mathfrak c$. Clearly, $(f \otimes 1)( x) \in A[\mathfrak
q]$. So $(f \otimes 1)( x) \in H$, and thus $f \otimes 1 = \pi^*(b)$
for some $b \in \End_{\O_F}(B)$.

On the other hand, if $b \in \End_{\O_F}(B)$, then
$$
(\phi \otimes 1) \pi^*(b) = \pi_1 b \pi  \in \Hom_{\O_F}(A, A') =
\Hom(E, E') \otimes \O_F.
$$
\end{proof}

Since $\phi$ is an isomorphism away from $q$, one sees from the
lemma
$$
\End_{\O_F}(B) \otimes \mathbb Z_l \cong (\End(E) \otimes Z_l)
\otimes_{\mathbb Z} \O_F
$$
for all $l \ne q$ via $\pi^*$. We now study
\begin{equation}
\O_{B, q}=\End_{\O_F} (B) \otimes \mathbb Z_q = \End_{\O_F \otimes
\mathbb Z_q} T_q(B),
\end{equation}
where $T_q(B)$ is the Tate module of $B$ at $q$. We identify
\begin{equation}
F \hookrightarrow F_q=F_{\mathfrak q} \oplus F_{\mathfrak q'} \cong
\mathbb Q_q \oplus \mathbb Q_q,  \quad \sqrt D \mapsto (\sqrt D,
-\sqrt D)
\end{equation}
as fixed at the beginning of this section. Let $\{ e, f\}$ be a {\it
$\phi$-normal basis} of $T_q(E) \subset V_q(E)=T_q(E)\otimes \mathbb
Q_q$ in the sense
\begin{equation}
T_q(E) = \mathbb Z_q e \oplus \mathbb Z_q f, \quad T_q(E')= \mathbb
  Z_q \phi(e) \oplus \mathbb Z q^{-1} \phi(f).
\end{equation}
To clear up notation, we view both $T_q(E)$ and $T_q(E')$ as
submodule of $V_q(E)=T_q(E) \otimes \mathbb Q_q$ so  that
$\phi(e)=e$ and $\phi(f)=f$.  Let $\mathfrak c_q= \mathfrak c\otimes
\mathbb Z_q =\mathbb Z_q (q^r,0) + \mathbb Z_q(0, q^s)$.
  It is easy to see that
  \begin{align*}
   T_q(A)&= T_q(E) \otimes_{\mathbb Z_q} \mathfrak c_q,
   \\
    T_q(A/A[\mathfrak q])&= T_q(A) \otimes_{\O_q} \mathfrak q_q^{-1} =T_q(E) \otimes_{\mathbb Z_q} \mathfrak c_q\mathfrak
    q_q^{-1},
    \\
    T_q(A')&= T_q(E'\otimes \mathfrak
    c)=T_q(E') \otimes_{\mathbb Z_q} \mathfrak c_q.
    \end{align*}
    and
    $$
    T_q(B) =T_q(A/A[\mathfrak q]) \cap  T_q(A').
    $$
    Now we use coordinates. Identify
    $$
    \O_q = \O_{\mathfrak q} \oplus \O_{\mathfrak q'} =\mathbb Z_q
    \oplus \mathbb Z_q,
$$
Then $\mathfrak c_q$ is generated by $(q^r, q^s)$ as an
$\O_q$-module, and $\mathfrak q_q$ is generated by $(q, 1)$ as an
$\O_q$-module. So
\begin{align*}
T_q(B)&= \left( \mathfrak c_q (q^{-1}, 1) e \oplus \mathfrak c_q
(q^{-1}, 1)f \right) \cap \left(\mathfrak c_q e \oplus \mathfrak
(q^{-1}, q^{-1})f\right)
\\
 &=\mathfrak c_q e \oplus \mathfrak c_q (q^{-1}, 1)f
 \\
  &= (\mathbb Z_q q^r e +\mathbb Z_q q^{r-1}f) \oplus (\mathbb Z_q
   q^s e \oplus \mathbb Z_q q^s f),
\end{align*}
and $(x, y) \in \O_q=\mathbb Z_q \oplus \mathbb Z_q$ acts on
$T_q(B)$ via
$$
(x, y) (a_1 q^r e+b_1 q^{r-1}f, a_2 q^s e+b_2 q^s f) =(xa_1 q^r e+
xb_1 q^{r-1}f, ya_2 q^s e+ yb_2 q^s f).
$$
So $\End_{\O_q}T_q(B)$ consists of $(\alpha, \beta) \in (\End
V_q(E))^2$ satisfying
\begin{equation}
\alpha (\mathbb Z_q q^r e +\mathbb Z_q q^{r-1}f) \subset \mathbb Z_q
q^r e +\mathbb Z_q q^{r-1}f, \quad \beta(\mathbb Z_q
   q^s e \oplus \mathbb Z_q q^s f) \subset \mathbb Z_q
   q^s e \oplus \mathbb Z_q q^s f.
   \end{equation}
   Here $V_q(E)=T_q(E) \otimes \mathbb Q_q =\mathbb Q_q e \oplus
   \mathbb Q_q f$. This is the same as $\alpha \in \End(T_q(E'))$
   and $\beta \in \End(T_q(E))$.
   So we have proved that

\begin{prop} Under the identification $\O_q=\O_{\mathfrak q}\oplus
\O_{\mathfrak q'} =\mathbb Z_q\oplus \mathbb Z_q$, one has
$$
\pi^* \End_{\O_q}(T_q(B)) =\{ (\alpha, \beta) \in (\phi^{-1}
\Hom(T_q(E), T_q(E')))^2:\, \phi \alpha \phi^{-1}  \in \End T_q(E'),
\quad \beta \in \End T_q(E)\}.
$$
Equivalently, with respect to a $\phi$-normal basis $\{e , f\}$,
  the  matrices
 of  $\alpha$  and $\beta$, still denoted by $\alpha$ and $\beta$ respectively, have the properties
  \begin{equation} \alpha \in \kzxz {\mathbb Z_q} { \frac{1}q \mathbb Z_q
  }
{q \mathbb Z_q } {\mathbb Z_q }, \quad \beta \in M_2(\mathbb Z_q),
\end{equation}
i.e.,
$$
\alpha \left(\substack{ e \\ f }\right) = \kzxz {x_1} {\frac{1}q y_1
} {q z_1} {w_1}\left(\substack{ e \\ f }\right), \quad \beta
\left(\substack{ e \\ f }\right)  = \kzxz {x_2} { y_2 } {z_2} {w_2}
\left(\substack{ e \\ f }\right),
$$
with $x_i, y_i, z_i, w_i \in  \mathbb Z_q$.
\end{prop}

\begin{cor} \label{cor2.4} One has
\begin{equation}
\pi^* \End_{\O_F}(B) = \{ \alpha +  \beta \otimes \frac{D+\sqrt D}2:
\,  \alpha, \beta \in \phi^{-1} \Hom(E, E')) \hbox{ satisfies }
(*_q) \hbox{ below } \}.
\end{equation}
Here  the matrices of $\alpha$ and $\beta$ with respect to a
$\phi$-normal basis of $T_q(E)$, still denoted by $\alpha$ and
$\beta$ respectively, have the following property $(*_q)$
\begin{equation} \label{eqQ}
\alpha + \beta  \frac{D +\sqrt D}2 \in \begin{pmatrix} \mathbb Z_q & \frac{1}q \mathbb Z_q \\
q \mathbb Z_q &\mathbb Z_q \end{pmatrix}, \quad \alpha + \beta
\frac{D-\sqrt D}2 \in M_2(\mathbb Z_q). \tag{$*_q$}
\end{equation}
(\ref{eqQ}) is equivalent to the condition
\begin{equation} \label{eqQ1}
\alpha + \beta  \frac{D +\sqrt D}2 \in \End(T_q(E')), \quad \alpha +
\beta \frac{D-\sqrt D}2 \in  \End (T_q(E)).
\end{equation}

  \end{cor}

\section{Local Intersection index} \label{sect3}

Let the notation and assumption be as in Section \ref{sect2}.
 The purpose of this section is to compute the local intersection index $i_p(\phi, \iota)$ in $(\ref{neweq2.6})$.
 We need a little preparation. Replacing $\Delta$ by $m \Delta$ in \cite[Lemma 4.1]{m=1}, one has

\begin{lem}   \label{lemold1.1} Let $m \ge 1$ be an integer and  let $0 <  n < m \sqrt{\tilde D}$ be an integer with
  $\frac{m^2 \tilde D -n^2}{D} \in \mathbb Z_{>0}$.

  (1) \quad When $D\nmid n$, there is a unique sign $\mu =\pm 1$ and a
  unique $2\times 2$  positive definite matrix $T_m(\mu n)
  =\kzxz {a} {b} {b} {c} \in \frac{1}m \Sym_{2}(\mathbb Z)$ such that
  \begin{align}
   \det T_m(\mu n) &=ac -b^2 =\frac{m^2 \tilde D -n^2}{Dm^2},
   \label{eqold1.1}
   \\
     \Delta &= \frac{2 \mu n_1 -Dc  -(2b +Dc) \sqrt D}{2},
    \label{eqold1.2}
\\
    -\mu n_1  &= a + Db + \frac{D^2-D}{4} c.  \label{eqold1.3}
    \end{align}
Here $n_1 =n/m$.

 (2) \quad  When $D |n$, for every sign $\mu =\pm 1$ there is a unique $2\times 2$ integral positive definite matrix $T_m(\mu n)
  =\kzxz {a} {b} {b} {c}$  satisfying the
  above conditions.
  \end{lem}

\begin{rem} \label{remI2.2}  Throughout this paper, the sum $\sum_{\mu}$ means either
$\sum_{ \mu =\pm 1}$ when $D|n$ or the unique term $\mu$ satisfying
the condition in Lemma \ref{lemold1.1} when $D \nmid n$.
\end{rem}

Notice that (\ref{eqold1.2}) implies
\begin{equation} \label{eqold1.4}
  2 \mu n_1 - Dc , \quad 2b + Dc \in \mathbb Z.
  \end{equation}

Now let $p \ne q$ be a prime, and let $\phi: E \rightarrow E'$  be a
cyclic isogeny of degree $q$ of supersingular elliptic curves over
$\bar{\mathbb F}_p$,  i.e., $(\phi: E \rightarrow E') \in \mathcal
Y_0(q) (\bar{\mathbb F}_p)$. We consider the set $I(\phi)$ of
$\O_K$-actions
$$
\iota: \O_K \hookrightarrow \End_{\O_F}(B)
$$
such that the Rosati involution associated to $\lambda$ gives the
complex conjugation on $K$ (as in Section \ref{sect2}). Set
\begin{align}
\pi^* \iota (\frac{w+\sqrt\Delta}2) &= \alpha_0 + \beta_0 \frac{D
+\sqrt D}2,\quad  \alpha_0, \beta_0 \in \phi^{-1} \Hom(E, E')
\label{eqold1.5}
\\
 \pi^* \iota(\sqrt\Delta) &= \alpha + \beta \frac{D+\sqrt D}2=x_1 +x_2
 \sqrt D,   \label{eqold1.6}
 \end{align}
with
\begin{equation} \label{eqold1.7}
 \alpha =2 \alpha_0 -w_0, \quad \beta = 2 \beta_0 -w_1,
\end{equation}
and
\begin{equation}
x_1 = \alpha + \frac{D}2 \beta,  \quad  x_2 = \frac{1}2 \beta.
\end{equation}
Let $\O_E=\End(E)$ and $\mathbb B=\O_E \otimes \mathbb Q$,
\begin{equation} \label{eq3.9}
V=\{ x \in \mathbb B:\, \tr x =0 \},  \quad Q(x)  =
-x^2
\end{equation}
and let \begin{equation}  \label{eq3.10} L(\phi)=(\mathbb Z + 2
\phi^{-1} \Hom(E, E')) \cap V.\end{equation}  Then $ \alpha, \beta
\in L(\phi) $.

  Notice that $(V, Q)$ is a quadratic subspace of the quadratic space $(\mathbb B,
  \det)$ where $\det(x)$ is the reduced norm of $x$. For
  $\vec{x}=(x_1, x_2, \cdots, x_n) \in \mathbb B^n$, we write
  \begin{equation}
  T(\vec{x} ) =\frac{1}2 (\vec x, \vec x)= \frac{1}2 ( (x_i,
  x_j)).
  \end{equation}

Let $\mathbb T(\phi)$ be the set of pairs $(\alpha, \beta) \in L(\phi)^2$ which satisfies (\ref{eqQ}) and $T(\alpha,
\beta) =T_q(\mu n)$ for some integer (unique) $0< n < q \sqrt{\tilde
D}$ with $\frac{q^2 \tilde D -n^2}{4D} \in  p \mathbb Z_{>0}$  and
some  sign (unique) $\mu =\pm 1$.

Let $\tilde{\mathbb T}(\phi)$ be the set of pairs $(\alpha_0, \beta_0)
\in (\phi^{-1}\Hom(E, E'))^2$ which satisfies (\ref{eqQ}) and $T(1,
\alpha_0, \beta_0) =\tilde T_q(\mu n)$ for some integer $0< n < q
\sqrt{\tilde D}$  with $\frac{q^2 \tilde D -n^2}{4D} \in  p \mathbb
Z_{>0}$ and some sign $\mu =\pm 1$. Here
\begin{equation}
\tilde T =\begin{pmatrix} 1 &0 &0 \\ \frac{w_0}2 &\frac{1}2 &0 \\
\frac{w_1}2 &0 &\frac{1}2\end{pmatrix} \diag(1, T)
\begin{pmatrix} 1 &\frac{w_1}2 &\frac{w_1}2 \\ 0 &\frac{1}2 &0 \\
0 &0 &\frac{1}2\end{pmatrix} =\begin{pmatrix}
 1&\frac{w_1}2 &\frac{w_1}2 \\
 \frac{w_0}2 & \frac{1}4 (a+w_0^2) &\frac{1}4 (b+ w_0 w_1)
 \\
 \frac{w_1}2  &\frac{1}4 (b+ w_0 w_1) &\frac{1}4 (c +w_1^2)
 \end{pmatrix}
\end{equation}
 and  $w =w_0 + w_1 \frac{D+\sqrt D}2$ is given in
(\ref{eqOK}).

\begin{prop}  \label{prop3.2} The correspondences
$$\iota \in I(\phi) \leftrightarrow
(\alpha, \beta) \in \mathbb T(\phi) \leftrightarrow (\alpha_0,
\beta_0) \in \tilde{\mathbb T}(\phi)
$$
via (\ref{eqold1.5})-(\ref{eqold1.7}) give bijections among
$I(\phi)$, $\mathbb T(\phi)$, and $\tilde{\mathbb  T}(\phi)$.
\end{prop}
\begin{proof} Given $\iota \in  I(\phi)$, and let $\alpha$ and
$\beta$ be given via (\ref{eqold1.6}). Then $(\alpha, \beta) \in
L(\phi)^2$ and satisfies (\ref{eqQ}).
 Write
$T(\alpha, \beta) =\kzxz {a} {b} {b} {c} $ with $a =\frac{1}2
(\alpha, \alpha) =-\alpha^2$, $b =\frac{1}2(\alpha, \beta)$, and $c
=\frac{1}2(\beta, \beta)= -\beta^2$. First,
\begin{align*}
\Delta &= (\pi^*\iota(\sqrt\Delta))^2 =(\alpha+\frac{D}2 \beta)^2
-(\alpha +\frac{D}2 \beta, \frac{1}2 \beta ) \sqrt D
\\
 &= -a -Db -\frac{D^2+D}4 c - (b+\frac{1}2 Dc) \sqrt D.
 \end{align*}
 We define $n=q n_1>0$ and $\mu =\pm 1$ by
 $$
 -\mu n_1 = a + Db +\frac{D^2 -D}4 c.
 $$
 Then
 $$
 \Delta = \frac{2 \mu n_1-Dc -(2b +Dc)\sqrt D}2
 $$
satisfying (\ref{eqold1.2}) in Lemma \ref{lemold1.1}. Now a simple
calculation using $\tilde D =\Delta \Delta'$ gives
$$
\det T(\alpha, \beta) =ac -b^2 =\frac{q^2\tilde D -n^2}{q^2D} $$
satisfying  (\ref{eqold1.1}). So $T(\alpha, \beta) =T_q(\mu n)$ for
a unique integer $n$ and a unique sign $\mu$  satisfying the
conditions in Lemma \ref{lemold1.1}. To show $p |q^2 \det T_q(\mu n)
=\frac{q^2 \tilde D -n^2}{D}$, we work over $\mathbb Z_p$ to avoid
the denominator $q$ in $\det T_q(\mu n)$. Write $L_p=L(\phi) \otimes
\mathbb Z_p$, and $\O_p=\O_E \otimes \mathbb Z_p$, then
$$
L_p= (\mathbb Z_p + 2 \O_p) \cap (V\otimes \mathbb Q_p)
$$
has determinant $4 p^2$. Let $$ \gamma = (\alpha, \beta) + 2 \alpha
\beta \in L_p.
$$ Then
$$(\alpha, \gamma) =(\beta, \gamma)
=0, \quad (\gamma, \gamma) =2 (\alpha, \alpha) (\beta, \beta)-2
(\alpha, \beta)^2=8 \det T_q(\mu n) . $$ So the determinant of $\{
\alpha, \beta,\gamma \}$ is
$$
\det T(\alpha, \beta, \gamma)=\det \diag( T_q(\mu n),  4 \det
T_q(\mu n))
 = 4 \det T_q(\mu n)^2.
    $$
So  we have thus $p|\det T_q(\mu n)$ in $\mathbb Z_p$, i.e., $p
|\frac{q^2 \tilde D -n^2}{D}$. Similarly, to show $4|q^2 \det
T_q(\mu n)$,  we work over $\mathbb Z_2$. It is easier to look at
$\tilde T_q(\mu n) \in \Sym_3(\mathbb Z_2)^{\vee}$ (since $\alpha_0,
\beta_0 \in \O_E \otimes \mathbb Z_2$). It implies that
\begin{equation} \label{eqI2.9}
a \equiv -w_0^2 \mod 4, \quad b \equiv -w_0 w_1 \mod 2, \quad c
\equiv -w_1^2 \mod 4.
\end{equation}
So $\det T_q(\mu n) = ac -b^2 \equiv  0\mod 4$, and therefore
$(\alpha, \beta)\in \mathbb T(\phi)$. A simple linear algebra
calculation shows that $(\alpha_0, \beta_0) \in \tilde{\mathbb
T}(\phi)$.

Next, we assume that $(\alpha , \beta) \in \mathbb T(\phi)$. Define
$\iota$ and $(\alpha_0, \beta_0)$ by (\ref{eqold1.6}) and
(\ref{eqold1.7}). The above calculation gives
$$
(\alpha +\beta \frac{D+\sqrt D}2)^2 =\Delta,
$$
so $\iota$ gives an embedding from $K$ into $\End_{\O_F}^0B $ such
that $\iota (\O_F[\sqrt\Delta]) \in \End_{\O_F}B$.  To show that
$\iota \in  I(\phi)$, it suffices to show $\alpha_0, \beta_0 \in
 \phi^{-1} \Hom(E, E')$.  Write by definition
$$
\alpha =-u_0 + 2 \alpha_1, \quad \beta=-u_1 + 2 \beta_1, \quad u=u_0
+ u_1 \frac{D+\sqrt D}2
$$
with $u_i \in \mathbb Z$, $\alpha_1, \beta_1\in \phi^{-1}\Hom(E,
E')$ . Then
$$
\pi^*\iota(\frac{u +\sqrt\Delta}2) = \alpha_1 + \beta_1\frac{D+\sqrt
D}2
$$
and $(\alpha_1, \beta_1) \in  (\phi^{-1}(E, E'))^2$ satisfies the
condition (\ref{eqQ}).  So $\iota(\frac{u +\sqrt\Delta}2) \in
\End_{\O_F}B$ and thus   $\frac{u +\sqrt\Delta}2 \in \O_K$. On the
other hand, $\frac{w+\sqrt\Delta}2 \in \O_K$.  So $\frac{u-w}2
 \in  \O_F$, i.e., $\frac{w_i -u_i}2 \in \mathbb Z$, and
 $$
 \alpha_0 =\alpha_1 +\frac{w_0 -u_1}2 \in \phi^{-1}(E, E'), \quad
 \beta_0 =\beta_1 +\frac{w_1 -u_1}2 \in \phi^{-1}(E, E')
 $$
 as claimed. So $(\alpha_0, \beta_0) \in \tilde{\mathbb T}(\phi)$ and
 $\iota \in I(\phi)$. Finally, if $(\alpha_0,
 \beta_0) \in \tilde{\mathbb T}(\phi)$, let $(\alpha, \beta)$ be given by (\ref{eqold1.7}). Then it is easy to check that  $(\alpha, \beta) \in \mathbb
 T(\phi)$.
\end{proof}

Now we are ready to compute local intersection indices.
\begin{prop}\label{prop3.4}  Let $\phi: E\rightarrow E'$ be an isogeny of supersingular elliptic curves over $\bF_p$  of degree $q$ ($p \ne q$).
Let $(\alpha, \beta) \in \mathbb T(\phi)$ be associated to $\iota
\in I(\phi)$, and let $T_q(\mu n) =T(\alpha, \beta)$ be the
associated matrix as in Proposition \ref{prop3.2}. Then
$$
i_p(\phi, \iota) = \frac{1}2 \left( \ord_p \frac{q^2 \tilde D
-n^2}{4D} +1\right)
$$
depends only on $n$.
\end{prop}
\begin{proof} This is a local question at $p$. $\iota \in  I(\phi)$
can be lifted to an embedding $\iota_I: \O_K \hookrightarrow
\End_{\O_F} (B_I)$ if and only if $\alpha_0$ and $\beta_0$ can be
lifted to $\alpha_{0, I}, \beta_{0, I} \in \phi_I^{-1} \Hom(E_I,
E'_E)$, which is equivalent to that $\phi$, $\phi\alpha_0$ and
$\phi\beta_0$ can be lifted to isogenies from $E_I$ to $E'_I$. So
$\iota_p(\phi, \iota)=i_p(\phi, \phi\alpha_0, \phi\beta_0)$ is the
local intersection index of $\phi, \phi\alpha_0, \phi\beta_0$
computed by Gross and Keating \cite{GK}. It depends only on
$T(\phi, \phi\alpha_0, \phi\beta_0)=qT_q(\mu n)$. The same
calculation as in \cite[Theorem 3.1]{m=1} (using Gross and Keating
's formula) gives (recall $n_1=n/q, p \ne q$)
$$
i_p(\phi, \iota) =\frac{1}2\left( \ord_p \frac{ \tilde D -n_1^2}{4D}
+1\right) =  \frac{1}2\left( \ord_p \frac{q^2 \tilde D -n^2}{4D}
+1\right)
$$
\end{proof}

So we have by (\ref{intersection}) and Proposition \ref{prop3.4}

\begin{theo} \label{theo3.5}  For $p \ne q$, one has
$$
(\mathcal T_q.\CM(K))_p = \frac{1}2 \sum_{\substack{ 0 < n < q
\sqrt{\tilde D} \\ \frac{q^2 \tilde D-n^2}{4D} \in p\mathbb Z_{>0}}}
\left( \ord_p \frac{q^2 \tilde D -n^2}{4D} +1\right) \sum_\mu
\sum_{\phi } \frac{R(\phi, T_q(\mu n))}{\#\hbox{Aut}(\phi)} .
$$
Here $R(\phi, T_q(\mu n))$ is the number of pairs $(\alpha, \beta)
\in L(\phi)^2$ such that $T(\alpha, \beta) =T_q(\mu n)$ and
$(\alpha, \beta)$ satisfies the condition (\ref{eqQ}), and
$\sum_\phi$ is over all isogenies (up to equivalence) $\phi : E
\rightarrow E'$ of supersingular elliptic curves over $\bF_p$  of
degree $q$ up to equivalence. Two isogenies $\phi_i: E_i \rightarrow
E_i'$ are equivalent if there isomorphisms $f: E_1 \cong E_2$ and
$f': E_1' \cong E_2'$ such $\phi_2 f = f' \phi_1$.
\end{theo}

\section{Local densities} \label{sect4}

We write $[\phi: E \rightarrow E']$ for the equivalence class of
$\phi$ and
\begin{equation}
\beta(p, \mu n) =  \sum_{[\phi: E \rightarrow E'] } \frac{R(\phi,
T_q(\mu n))}{\#\hbox{Aut}(\phi)} .
\end{equation}
One can show that $\beta(p, \mu n)$ is the $T_q(\mu n)$-th Fourier
coefficient of some Siegel-Eisenstein series of genus two and weight
$3/2$, and is thus product of local Whittaker functions, which are
slight generalization of local densities computed in
\cite{YaDensity1} and \cite{YaDensity2}. In principle, the idea in
\cite{YaDensity1} and \cite{YaDensity2} can be extended to handle
the general case. However, the actual computation is already
complicated in \cite{YaDensity1} and \cite{YaDensity2}. In this
section, we use a different way to write $\beta(p, \mu n)$ directly
as product of local integrals over quaternions. In next section, we
take advantage of known structure of quaternions to compute the
involved local integrals.

  Fix a cyclic isogeny $\phi_0: E_0 \rightarrow E_0'$ of
  supersingular elliptic curves (over $\bF_p$) of degree $q$.  and a
  $\phi_0$-normal basis $\{ e_0, f_0\}$ of the Tate module
  $T_q(E_0)$. Let $\O=\End(E_0)$ and $\mathbb B= \O \otimes \mathbb
  Q$ be the unique quaternion algebra over $\mathbb Q$ ramified exactly at $p$ and $\infty$. Let $(B_0, \iota_0, \lambda_0) \in \mathcal M(\bF_p)$ be the
  abelian surface with real multiplication associated to $\phi_0$.
  Let $V$ and $L(\phi_0)$ be the ternary quadratic space and lattice  defined  in (\ref{eq3.9}) and
  (\ref{eq3.10}) with $\phi$ replaced by $\phi_0$. For $l \ne q$,
  let
  \begin{equation}
  L_l =L(\phi_0)\otimes \mathbb Z_l, \quad \Psi_l =\cha
  (L_l^2).
  \end{equation}
  For $l=q$,  view $\mathbb B_q =\mathbb B \otimes \mathbb Q_q$
  as the endomorphism ring of $V_q(E_0)=T_q(E_0) \otimes \mathbb
  Q_q$ and identify it with $M_2(\mathbb Q_q)$  using the $\phi$-normal basis
  $\{e_0, f_0\}$. Under this identification, $\O_q = M_2(\mathbb
  Z_q)$. Let
 \begin{equation}
L_q' = \{ X=\kzxz {x} {\frac{1}q y}  {z} {-x} \in V_q: \, x, y, z
\in \mathbb Z_q \}
\end{equation}
and
\begin{equation}
\Omega_q= \{ \vec x = {}^t(X_1, X_2) \in (L_q')^2:\,  z_1 + z_2
\frac{D+\sqrt D}2 \equiv 0 \mod q,  \,
  y_1 + y_2
\frac{D-\sqrt D}2 \equiv 0 \mod q \}
\end{equation}
where  $X_i = \kzxz {x_i} {\frac{1}q y_i} {z_i} {-x_i} \in L_q'$.
Let
 \begin{equation}
 \Psi_q =\cha(\Omega_q), \quad \Psi= \otimes_{l <\infty} \Psi_l \in
 S(V(\A_f)^2).
 \end{equation}
Next, let $\mathcal K= \prod_{l < \infty} \mathcal K_l \subset
\mathbb B_f^*$ be the compact subgroup of $\mathbb B_f^*$ defined by
\begin{equation}
\mathcal K_l=\begin{cases} \O_l^*   &\ff l \ne q,
\\
  K_0(q)=\{ \kabcd \in M_2(\mathbb Z_q): c \equiv 0 \mod q \}  &\ff l =q.
  \end{cases}
  \end{equation}
Clearly, $\Psi$ is $\mathcal K$-invariant.  The main purpose of this
section is to prove

\begin{theo}  \label{theo4.1} Let the notation be as above. Then
\begin{equation} \label{eq4.7}
\beta(p, \mu n) =\frac{1}2  \int_{\mathbb Q_f^* \backslash \mathbb
B_f^* /\mathcal K} \Psi(g^{-1}.\vec x_0) dg
\end{equation}
if there is $\vec x_0=V(\A_f)^2$ with $T(\vec x_0) = T_q(\mu n)$.
Otherwise, $\beta(p, \mu n) =0$. Here
$$g.\vec x= (g.X_1, g.X_2) =(gX_1g^{-1}, gX_2g^{-1}), \quad \vec
x={}^t(X_1, X_2),
$$
and $dg$ is the Tamagawa  measure on $\mathbb B_f^*$.
\end{theo}

We first recall a close relation between $\mathbb B_f^*$ and cyclic
isogenies $\phi: E \rightarrow E'$ of degree $q$.  Let $T_l(E)$ be
the $l$-Tate module of $E$ for $l \ne p$ and let $T_p(E)$ be the
covariant Die\'udonne module of $E$ over the Witt ring $W=W(\bF_p)$,
and let $\hat T(E) = \otimes T_l(E)$. A homomorphism from $T_p(E)$
to $T_p(E')$ means a $W$-linear map on the Dieudonn\'e modules which
commute with the Frobenius map. Then for $b \in \mathbb B_f^*$,
there is an quasi-isogeny $f: E\rightarrow E_0$ such that $\hat T(f)
\hat T(E)=b \hat T(E_0)$. Moreover, the equivalence class of $f:
E\rightarrow E_0$ is determined by $b \mod \hat{\O}^*$ \cite[Section
2.4]{We1}. Choose an integer $n>0$ such that $nf$ is an isogeny. Let
$E'$ be the fiber product as shown in  the following diagram.

\begin{equation} \xymatrix{
 E\ar@{-->}[r]^-{\phi} \ar[d]^-{nf} &E'\ar[r]^-{\phi_1}  \ar[d]^-{nf'} &E\ar[d]^-{nf} \cr
 E_0\ar[r]^-{\phi_0} &E_0' \ar[r]^-{\phi_0'} &E_0\cr}
 \end{equation}

Then there is a unique $\phi: E \rightarrow E'$ making the above
diagram commute. Let $\mathcal S_0(q)$ be the set of equivalence
classes  $[\phi: E\rightarrow E', f, f']$ of the diagrams:
\begin{equation}
\xymatrix{
  E\ar[r]^-{\phi} \ar@{~>}[d]^-{f} &E'
  \ar@{~>}[d]^-{f'} \cr
   E_0 \ar[r]^-{\phi_0} &E_0' \cr }
   \end{equation}
where $E\rightsquigarrow E_0$ stands for quasi-isogeny. Here two
such diagrams are equivalent if there are isomorphisms $g: E_1
\rightarrow E_2$ and $g': E_1'\rightarrow E_2'$ such that the
following diagram commutes:
\begin{equation}
\xymatrix{ E_1 \ar[rr]^-{\phi_1} \ar@{~>}[dd]^-{f_1}
\ar@{-->}[dr]^-{g}& &E_1' \ar@{-->}[dr]^-{g'}\ar@{~>}[dd]^-{f_1'}
&\cr &E_2 \ar[rr]^-{\phi_2} \ar@{~>}[dl]^-{f_2} &
&E_2'\ar@{~>}[dl]^-{f_2'} \cr E_0\ar[rr]^-{\phi_0} & &E_0' &\cr}.
\end{equation}

Let  $S_0(q)$ be the set of equivalence classes $[\phi: E
\rightarrow E']$ of  degree $q$ isogenies of supersingular curves
over $\bar{\mathbb F}_p$. Then one has

\begin{prop}  \label{prop4.2} The map $b \in \mathbb B_f^* \mapsto [\phi:
E\rightarrow E', f, f']$ gives rise to a bijection between $\mathbb
B_f^*/\mathcal K$ and $\mathcal S_0(q)$. The map $b \in \mathbb
B_f^* \mapsto [\phi: E\rightarrow E']$ gives rise to a bijection
between $\mathbb B^* \backslash \mathbb B_f^*/\mathcal K$ and the
set $S_0(q)$.  Moreover, for $\alpha_0, \beta_0\in \mathbb
B=\End(E_0) \otimes \mathbb Q$, let $\alpha = f^{-1} \alpha_0 f,
\beta = f^{-1} \beta_0 f\in \End(E) \otimes \mathbb Q$. Then

(1) \quad $\alpha \in \End(E)$ if and only if $b^{-1} \alpha_0 b
\in \hat{\O} = \O \otimes \hat{\mathbb Z}$.

(2) \quad $\phi \alpha \phi^{-1} \in \End(E')$ if and only if
$\phi_0 b^{-1} \alpha_0 b  \phi_0^{-1} \in \End(E_0') \otimes
\hat{\mathbb Z}$.

(3) \quad $\alpha \in \End(\phi)$ if and only if $b^{-1} \alpha_0
b \in \End(\phi_0) \otimes  \hat{\mathbb Z}$.

(4) \quad  $\alpha + \beta\frac{D+\sqrt D}2 \in \pi^*
\End_{\O_F}(B)$ if and only if $ b^{-1} (\alpha_0 +\beta_0 \frac{D
+\sqrt D}2) b \in \pi_0^*
 (\End_{\O_F} (B_0) \otimes \hat{\mathbb Z}).
 $
\end{prop}
\begin{proof}   The same argument as in \cite[Section 2.4]{We1} gives the bijections.

(1) \quad Clearly, $\alpha \in \End(E)$ if and only if $\hat
T(\alpha) \hat T(E) \subset \hat T(E)$. If $b^{-1} \alpha_0 b \in
\hat\O$, then
\begin{align*}
 \hat T(\alpha) \hat T(E) &= \hat T(f)^{-1} \hat T(\alpha_0)  \hat
 T(f) \hat T(E)
  =\hat T(f)^{-1} b b^{-1} \alpha_0  b \hat T(E_0)
  \\
   &
  \subset \hat T(f)^{-1} b \hat T_0(E_0) =\hat T(f)^{-1} \hat T(f)
  \hat T(E) = \hat T(E), \end{align*}
and thus $\alpha \in \End(E)$. Here we identify $\alpha_0$ with
$\hat T(\alpha_0) \in  \End^0( \hat T(E))$.
 Reversing the procedure with $\alpha_0 = f \alpha f^{-1}$, one
 sees that $b^{-1} \alpha_0 b \in \hat\O$ if $\alpha \in \End(E)$.

(2) \quad Since
$$ \hat T(\phi \alpha \phi^{-1})
 =  \hat T(\phi f^{-1}) \hat T(\alpha_0) \hat T(f \phi^{-1})
 = \hat T(f')^{-1}  \hat T(\phi_0 \alpha_0 \phi_0^{-1}) \hat
 T(f'),
 $$
 the equivalence class of $E'\rightsquigarrow E_0'$ is associated
 to $b'=\phi_0 b \phi_0^{-1}$ when $E \rightsquigarrow E_0$ is
 associated to $b$. Now (2) follows from $(1)$.
 $(3)$ follows from $(1)$ and $(2)$ since $\alpha \in
\End(\phi)$ if and only if $\alpha \in \End(E)$ and $\phi \alpha
\phi^{-1} \in \End(E')$.

 (4) \quad Since
 $$
 \hat T(\phi  \alpha) = \hat T(f')  (\phi_0 b \phi_0^{-1})
 \phi_0(b^{-1}\alpha_0 b),
 $$
$ \alpha \in  \phi^{-1} \Hom(E, E')$ if and only if $b^{-1}
\alpha_0 b \in \phi_0^{-1} \Hom(\hat T(E), \hat T(E'))$. So by (1)
and (2) (more precisely their  local analogue at $q$) and
Corollary \ref{cor2.4}, one has
\begin{align*}
  &\alpha +\beta \frac{D +\sqrt D}{2} \in \pi^* \End_{\O_F}(B)
  \\
  &\Leftrightarrow  \alpha, \beta \in  \phi^{-1} \Hom(E, E')
  \hbox{ and } (\ref{eqQ1})
  \\
  &\Leftrightarrow b^{-1} \alpha_0 b, b^{-1} \beta_0 b \in \phi_0^{-1} \Hom(\hat T(E), \hat T(E')), \hbox{
and } (\ref{eqQ1}) \hbox{ for } (b^{-1} \alpha_0 b, b^{-1} \beta_0
b)
\\
 &\Leftrightarrow b^{-1} (\alpha_0 +\beta_0 \frac{D+\sqrt D}2)  b
 \in \pi_0^*( \End_{\O_F} (B_0) \otimes \hat{ \mathbb Z})
\end{align*}
as claimed.
\end{proof}

{\bf Proof of Theorem \ref{theo4.1}:} Let
\begin{equation}
f_{\mu n}(g) = \sum_{\substack{ \vec x \in V^2 \\  T(\vec x)
=T_q(\mu n)} } \Psi(g^{-1}.\vec x).
\end{equation}
Then $f_{\mu n} $ is left $\mathbb B^*$-invariant and right
$\mathcal K$-invariant. We claim
\begin{equation} \label{eqold4.9}
\beta(p, \mu n) = \int_{\mathbb B^* \backslash \mathbb
B_f^*/\mathcal K} f_{\mu n}(g) dg .
\end{equation}
Indeed, write $\mathbb B_f^* = \bigsqcup_j \mathbb B^* b_j \mathcal
K$ with $b_j \in \mathbb B_f^*$, and let $[\phi_i: E_i \rightarrow
E_i'] \in S_0(q)$ be the associated equivalence class of cyclic
isogenies as given in Proposition \ref{prop4.2}. Since the map
$$
\mathbb B^* \times \mathcal K \rightarrow  \mathbb B^* b_j\mathcal
K,  \quad (b, k) \mapsto bb_j k
$$
has fiber $\mathbb B^* \cap b_j \mathcal K b_j^{-1}$ at $b_j$, one
has
$$
\int_{\mathbb B^* \backslash \mathbb B_f^*/\mathcal K} f_{\mu n}(g)
dg =\sum_j f(b_j)\int_{\mathbb B\backslash \mathbb B b_j \mathcal
K/\mathcal K} dg= \sum_{j} \frac{1}{\#\mathbb B^* \cap b_j \mathcal
K b_j^{-1}} f_{\mu n}(b_j).
$$
Let $[\phi_j: E_j\rightarrow E_j'] \in S_0(q)$ be associated to
 $b_j$, and choose $f_j: E_j \rightsquigarrow E_0$ and $f_j' \rightsquigarrow E_0'$
  so that $[\phi_j: E_j\rightarrow E_j', f_j, f_j'] \in \mathcal S_0(q)$ is
  associated to $b_j$ by Proposition \ref{prop4.2}. For $\vec x ={}^t(\delta_0, \beta_0) \in V^2$ with $T(\vec
 x) =T_q(\mu n)$, one has by definition
$\Psi(\vec x) =1$ if and only if $\delta_0 + \beta_0 \frac{D+\sqrt
D}2 \in \pi_0^*(\End_{\O_F}(B_0))$, and for $\vec x ={}^t(\delta_0,
\beta_0) \in V(\A_f)^2$, $\Psi(\vec x) =1$ if and only if $\delta_0
+ \beta_0 \frac{D+\sqrt D}2 \in \pi_0^*(\End_{\O_F}(B_0) \otimes
\hat{\mathbb Z})$. So
  one has by Proposition \ref{prop4.2}
 \begin{align*}
  \Psi(b_j^{-1}.\vec x) =1 \Leftrightarrow \delta_j + \beta_j
  \frac{D+\sqrt D}2 \in \pi^* \End_{\O_F}(B_j)
  \end{align*}
  where $\delta_j =f^{-1} \delta_0 f_j$ and $\beta_j =f_j^{-1} \delta_0
  f_j$. So
   $$
    f_{\mu n}(b_j) = R(\phi_j, T_q(\mu n)).
    $$
   Next for $\delta_0 \in \mathbb B^*$, one has by Proposition
  \ref{prop4.2}
  \begin{align*}
  \delta_0 \in \mathbb B^* \cap b_j \mathcal K b_j^{-1}
   &\Leftrightarrow b_j^{-1} \delta_0 b_j \in \mathcal K
   =(\End(\phi_0) \otimes \hat{\mathbb Z})^*
   \\
    &\Leftrightarrow \delta =f_j^{-1} \delta_0 f_j \in \Aut(\phi_j).
    \end{align*}
    So $\# \mathbb B^* \cap b_j \mathcal K b_j^{-1} =\#\Aut(\phi_j)$,
    and thus
\begin{align*}
\int_{\mathbb B^* \backslash \mathbb B_f^*/\mathcal K} f_{\mu n}(g)
dg &= \sum_{j} \frac{1}{\#\mathbb B^* \cap b_j \mathcal K b_j^{-1}}
f_{\mu n}(b_j)
\\
 &=\sum_j \frac{1}{\# \Aut(\phi_j)} R(\phi_j, T_q(\mu n))
 \\
  &= \beta(p, \mu n)
\end{align*}
by Proposition \ref{prop4.2}. This proves claim (\ref{eqold4.9}).
If there is no $\vec x \in V^2$ such that $T(\vec x) =T_q(\mu n)$,
one has clearly $\beta(p, \mu n) =0$ by (\ref{eqold4.9}). At the
same time, the Hasse principle asserts that there is no $\vec x \in
V(\A_f)^2$ with $T(\vec x) =T_q(\mu n)$, and thus the right hand
side of (\ref{eq4.7}) is zero too, Theorem \ref{theo4.1} holds
trivially in this case. Now assume there is a  $\vec x \in V^2$ such
that $T(\vec x) =T_q(\mu n)$, and  choose  such a vector  $\vec
x_0$. By Witt's theorem, for any $\vec x \in V^2$ with $T(\vec x)
=T_q(\mu n)$, there is $b \in \mathbb B^*$ such that $b^{-1}. \vec
x_0 =\vec x$. It is easy to check that the stabilizer of $\vec x_0$
in $\mathbb B^*$ is $\mathbb Q^*$.  So we have
\begin{align*}
\int_{\mathbb B^* \backslash \mathbb B_f^*/\mathcal K} f_{\mu n}(g)
dg &=\int_{\mathbb B^* \backslash \mathbb B_f^*/\mathcal K}
 \sum_{b\in \mathbb Q^* \backslash \mathbb B^*} \Psi( (bg)^{-1}.\vec
 x_0) dg
 \\
  &=\int_{\mathbb Q^* \backslash \mathbb B_f^*/\mathcal K} \Psi(g^{-1}.\vec x_0)
dg
\\
  &= \int_{\mathbb Q^* \backslash \mathbb Q_f^*} d^*x \cdot
   \int_{\mathbb Q_f^* \backslash \mathbb B_f^*/\mathcal K} \Psi(g^{-1}.\vec x_0)
dg.
\end{align*}
Here $d^*x$ is the Haar measure on $\mathbb Q_f^* =\A_f^*$ such that
$\hat{\mathbb  Z^*}$ has Haar measure $1$. Now Theorem \ref{theo4.1}
follows from the well-known fact
$$
\int_{\mathbb Q^* \backslash \mathbb Q_f^*} d^*x =\frac{1}2,
$$
since $\Q_f^* =\mathbb Q^* \hat{\mathbb Z^*}$ and $\mathbb Q^* \cap
\hat{\mathbb Z^*} =\{ \pm 1\}$.

\section{Local computation} \label{sect5}

Let the notation be as in Section \ref{sect4}.  The main purpose of
this section is to compute the local integrals
\begin{equation}
\beta_l(T_q(\mu n), \Psi_l) = \int_{\mathbb Q_l^* \backslash \mathbb
B_l^*/\mathcal K_l} \Psi_l(h^{-1}.\vec x_0) dh
\end{equation}
 where $\vec x_0 \in V_l^2$ with $ T(\vec x_0)=T_q(\mu n)$, and $dh$ is
 a Haar measure on $\mathbb B_l^*$. It is a long calculation
 for $l=q$ and is quite  technical. We summarize the result as two
 separate theorems for the convenience of the reader. Theorem
 \ref{theo5.1} will be restated as Propositions \ref{oldprop5.3} and
 \ref{oldprop5.4}, while Theorem \ref{theo5.2} will be restated as
 Propositions \ref{oldprop5.5}, \ref{oldprop5.9}, and
 \ref{oldprop5.10}

 \begin{theo} \label{theo5.1}  For $l \ne q$, $T_q(\mu n)$ is
 $\mathbb Z_l$-equivalent to $\diag(\alpha_l, \alpha_l^{-1} \det
 T_q(\mu n))$ with $\alpha_l \in \mathbb Z_l^*$. Let $t_l = \ord_l
 \frac{q^2 \tilde D -n^2}{4D q^2}$. Then
 $$
\beta_l(T_q(\mu n), \Psi_l)=\begin{cases}
 1 -(-\alpha_p, p)_p^{t_p} &\ff l =p,
 \\
 \frac{1+(-1)^{t_l}}2 &\ff l \ne p, (-\alpha_l, l)_l=-1,
 \\
  t_l +1 &\ff l\ne p, (-\alpha_l, l)_l =1.
  \end{cases}
  $$
\end{theo}
\begin{theo} \label{theo5.2}

(1) \quad If $q \nmid n$, then $\beta_q(T_q(\mu n), \Psi_q) =1$.

(2) \quad If $q|n$ and $t_q =\ord_q
 \frac{q^2 \tilde D -n^2}{4D q^2} =0$, then
$$
\beta_q(T_q(\mu n, \Psi_q) =\begin{cases}
  4 &\ff q \hbox{ split completely in } \tilde K,
  \\
   2 &\ff q \hbox{ inert in } \tilde F, q\O_{\tilde F} \hbox{ split
    in } \tilde K,
    \\
    0  &\hbox{otherwise}.
    \end{cases}
    $$

    (3) \quad If $q|n$ and $t_q =\ord_q
 \frac{q^2 \tilde D -n^2}{4D q^2} >0$, then $T_q(\mu n)$ is $\mathbb
 Z_q$-equivalent to $\diag(\alpha_q, \alpha_q^{-1} \det T_q(\mu n))$
 with $\alpha_q \in \mathbb Z_q^*$, and
$$
\beta_q(T_q(\mu  n),\Psi_q) =\begin{cases}
  0 &\ff (-\alpha_q, q)_q=-1,
  \\
  2(t_q+2) &\ff (-\alpha_q, q)_q =1.
  \end{cases}
  $$
\end{theo}

 For any locally constant function  with compact support $f \in
 S(V_l^2)$ and a non-degenerate symmetric $2 \times 2$ matrix $T$
 over $\mathbb Q_l$, let
\begin{equation}
\gamma_l(T, f) =\int_{\mathbb Q_l^* \backslash \mathbb B_l^*}
f(h^{-1}.\vec x_0) dh
\end{equation}
with $T(\vec x_0) =T$. Then
\begin{equation}
\beta_l (T_q(\mu n), \Psi_l) = \frac{1}{\vol(K_l)} \gamma_l(T_q(\mu
n), \Psi_l).
 \end{equation}
Notice that $\beta_l$ is independent of the choice of the Haar
measure while $\gamma_l$ gives freedom of the choice of $f \in
S(V_l^2)$. We first give some general comments and lemmas.

When $l \ne p$, $\mathbb B_l^* =\GL_2(\mathbb Q_l)$ has two actions on
$V_l^2$, the orthogonal  action (by conjugation)
$$
h. {}^t(X_1, X_2) = {}^t(h X_1 h^{-1}, hX_2h^{-1})
$$
and the natural linear action
$$
\begin{pmatrix}  {g_1} &{g_2} \\ {g_3} &{g_4} \end{pmatrix} \begin{pmatrix} X_1 \\ X_2
\end{pmatrix} = \begin{pmatrix} g_1 X_1 +g_2 X_2 \\  g_3 X_1 + g_4
X_2 \end{pmatrix}
$$
To distinguish them, we write the orthogonal action as $h.x$. We
also have the linear action of $\GL_2(\mathbb Q_p)$ on $V_p^2$ while
$\mathbb B_p^*$ acts on $V_p^2$ orthogonally (by conjugation). These
two actions commute. This commutativity implies the following lemma
easily.

\begin{lem} \label{oldlem5.1}  Let  $T =g \tilde T \, {}^tg$ with $g
\in \GL_2(\mathbb Q_l)$. Then for any $f \in S(V_l^2)$
$$
\gamma_l(T, f) = \gamma_l(\tilde T, f_{g^{-1}})
$$
where $f_g(\vec x) = f(g^{-1} \vec x)$.
\end{lem}

The following lemma is well-known.

\begin{lem} \label{oldlem5.2} Write $h(r, u) = \kzxz {l^r} {u} {0} {1} $ and $h'(r, u)
= h(r, u) \kzxz {0} {1} {1} {0}$ for $ r \in \mathbb Z$ and $u \in
\mathbb Q_l$. Then

$$
\mathbb Q_l^* \backslash  \GL_2(\mathbb Q_l) = \bigcup_{ r \in
\mathbb Z, u \hbox{mod } l^r } h(r, u) \GL_2(\mathbb Z_l),
$$

$$
\mathbb Q_q^* \backslash  \GL_2(\mathbb Q_q) = \bigcup_{r \in
\mathbb Z, u \hbox{mod }l^r} h(r, u) K_0(q) \bigcup (\bigcup_{r \in
\mathbb Z, u \hbox{mod } l^{r+1}} h'(r, u) K_0(q)),
$$
and
$$
\mathbb Q_p^*\backslash  \mathbb B_p^* = \O_p^* \cup \pi \O_p^*
$$
where $\pi \in \mathbb B_p^*$ with $\pi^2 =p$.
\end{lem}

\subsection{The case $l \nmid pq$}

\begin{prop} \label{oldprop5.3} For $l \nmid p q$, $T_q(\mu n)$ is $\mathbb
Z_l$-equivalent to $\diag(\alpha_l, \alpha_l^{-1} \det T_q(\mu n))$
with $\alpha_l \in \mathbb Z_l^*$. Let $t_l=\ord_l \det T_q(\mu n) =
\ord_{l} \frac{q^2 \tilde D -n^2}{4 Dq^2} $. Then
$$
\beta_l(T_q(\mu n), \Psi_l) =\begin{cases}
 \frac{1+(-1)^{t_l}}2 &\ff (-\alpha_l, l)_l=-1,
 \\
  t_l + 1   &\ff (-\alpha_l, l)_l=1.
  \end{cases}
  $$
  \end{prop}
\begin{proof} Write  $T_q(\mu n) =g \diag (\alpha_l, \alpha_l^{-1} \det T_q(\mu n))
{}^t g$ with some $g \in \GL_2(\mathbb Z_l)$. Since $\Psi_l$ is
$\GL_2(\mathbb Z_l)$-invariant under the linear action, $(\Psi_l)_g
= \Psi_l$. So Lemma \ref{oldlem5.1} implies
$$
\beta_l(T_q(\mu n), \Psi_l) = \beta_l(\diag (\alpha_l, \alpha_l^{-1}
\det T_q(\mu n)), \Psi_l).
$$
In general, for $T = \diag(\epsilon_1, \epsilon_2 l^t)$ with
$\epsilon_i \in \mathbb Z_l^*$, $t \in \mathbb Z_{\ge 0}$, and
$(-\epsilon_1, -\epsilon_2)_l=1$ (it is only a condition for $l=2$
and is true in our case $(\alpha_l, \alpha_l^{-1} \det T_q(\mu n))$
\cite[Lemma 4.1]{m=1}), let
\begin{equation}
X_1 = \kzxz {0} {1} {-\epsilon_1} {0} \in L_l, \quad  Q(X_1)
=\epsilon_1,
\end{equation}
Then
\begin{equation}
(\mathbb Q_l X_1)^\perp = \{ \kzxz {x} {y} {\epsilon_1 y} {-x} \in
V_l: \, x, y \in \mathbb Q_l \}.
\end{equation}
So there is $\vec x={}^t(X_1, X_2) \in V_l^2$ with $T(\vec x)= T$ if
and only if there are $x, y \in \mathbb Q_l$ such that
\begin{equation} \label{oldeq5.9}
x^2 + \epsilon_1 y^2 =-\epsilon_2 l^t,
\end{equation}
which is equivalent to $(-\epsilon_1, -\epsilon_2 l^t)_l=1$, i.e.,
\begin{equation} \label{oldeq5.10} (-\epsilon_1, l)_l^t =1.
\end{equation}
Assume (\ref{oldeq5.10}) and $l \ne 2$. When $(-\epsilon_1, l)_l=-1$
and $t$ even, $(\ref{oldeq5.9})$ has a solution $x_0, y_0 \in
l^{\frac{t}2} \mathbb Z_l^*$. When $(-\epsilon_1, l)_l=1$,
(\ref{oldeq5.9}) has a solution $x_0, y_0 \in \mathbb Z_l^*$. Fix
such a solution, and let
\begin{equation} \label{oldeq5.11}
X_2 =\kzxz {x_0} {y_0} {\epsilon_1 y_0} {-x_0},\quad   \vec
x_0={}^t(X_1, X_2) \in L_l^2,
\end{equation}
with $T(\vec x_0) = T$. A simple calculation gives
\begin{align}
h(r, u)^{-1}.X_1 &= \kzxz {\epsilon_1 u} {l^{-r} (1+ \epsilon_1
u^2)} {-\epsilon_1 l^r} {-\epsilon_1 u}
\\
h(r, u)^{-1}.X_2 &=\kzxz {x_0 -\epsilon_1 y_0 u} {l^{-r}(y_0 +2 x_0
u -\epsilon_1 y_0 u^2)} {\epsilon_1 y_0 l^r} {-x_0 + \epsilon_1 y_0
u}
\end{align}
So $h(r, u)^{-1}.\vec x_0 \in L_l^2$ if and only if
$$
r \ge 0, u \in \mathbb Z_l,\quad  1+ \epsilon_1 u^2 \equiv 0 \mod
l^r,\quad y_0 +2 x_0 u -\epsilon_1 y_0 u^2 \equiv 0 \mod l^r,
$$
or equivalently,
\begin{equation} \label{oldeq5.14}
r \ge 0, b \in \mathbb Z_l, \quad x_0 u + y_0 \equiv 0 \mod
l^r,\quad 1+ \epsilon_1 u^2 \equiv 0 \mod  l^r.
\end{equation}

{\bf Case 1}: First we assume $(-\epsilon_1, l)_l=-1$ and $t$ is
even. In this case one has always $1 + \epsilon_1 u^2 \in \mathbb
Z_l^*$, and thus $r=0$ and $u \in \mathbb Z_l$, i.e., $h(0, u)\in
\mathcal K_l=\GL_2(\mathbb Z_l)$ is the only coset with $h(r,
u)^{-1}.\vec x_0 \in L_l^2$, i.e., $\Psi_l(h(r, u).\vec x_0)\ne 0$.
So $ \beta_l(T, \Psi_l) = 1 $ in this case.

{\bf Case 2}: Now we assume $(\epsilon_1, l)_l=1$. Using
(\ref{oldeq5.14}), one has
$$
x_0^2 (1+ \epsilon_1 u^2) \equiv x_0^2 + \epsilon_1 y_0^2
=-\epsilon_2 l^t \mod l^r
$$
and so $0 \le r \le t$. Moreover, for $0 \le r \le t$, the above
condition also shows that $1 + \epsilon_1 u^2 \equiv 0 \mod l^r$
follows from $u \equiv -\frac{y_0}{x_0} \mod l^r$.  This implies
\begin{align*}
\beta_l(T, \Psi_l)&= \sum_{r \in \mathbb Z, u \hbox{mod }l^r }
\Psi(h(r, u)^{-1}.\vec x_0)
\\
 &=\sum_{0 \le r \le t, u=-y_0/x_0 \hbox{mod }l^r} 1 =t+1.
 \end{align*}
This proves the proposition for $l \ne 2$. This case $l=2$ is
similar with some modification, including
$$
L_2 =\{ A \in \mathbb Z_2 + 2 M_2(\mathbb Z_2):\, \tr A =0\} =\{
\kzxz {x} {2y} {2z} {-x}:\, x, y, z \in \mathbb Z_2 \}.
$$
We leave the detail to the reader.
\end{proof}

\subsection{The case $l=p$}

\begin{prop} \label{oldprop5.4}  For $l=p$, $T_p(\mu n)$ is $\mathbb Z_p$-equivalent to
$\diag(\alpha_p, \alpha_p^{-1} \det T_q(\mu n))$ with $\alpha_p \in
\mathbb Z_p^*$, and
$$
\beta_p(T_q(\mu n), \Psi_p) = 1 -(-\alpha_p, p)_p^{t_p} .
$$
\end{prop}

\begin{proof} We first assume that $p\ne 2$. Recall that $\O_p$ is the maximal order of $\mathbb
B_p$ and is consisting of elements of integral reduced norm. So
$$
L_p =(\mathbb Z_p + 2 \O_p)\cap V_p=  \{ x \in V_p:\, Q(x) =-x^2 \in
\mathbb Z_p \}
$$
has a basis $\{e, \pi, \pi e\}$ with $e^2 =a \in \mathbb Z_p^*$,
$\pi^2 =p$, and $\pi e =- e \pi$ with $(a, p)_p=-1$.  Since $\Psi_p$
is $\GL_2(\mathbb Z_p)$-invariant (linearly), Lemma \ref{oldlem5.1}
implies that
$$
\beta_p(T_q(\mu n), \Psi_p) = \beta_p(\diag(\alpha_p,
\alpha_p^{-1} \det T_q(\mu n)), \Psi_p).
$$
For $T =\diag (\epsilon_1, \epsilon_2 p^t)$ with $\epsilon_i \in
\mathbb Z_p^*$ and $ t \in \mathbb Z_{\ge 0}$ and $(-\epsilon_1,
-\epsilon_2)_p=1$, the above comment implies that if $T(\vec x) =T$
for some $\vec x \in V_p^2$, then $\vec x \in L_p^2$. If $X=x_1 e +
x_2 \pi + x_3 \pi e$ satisfies
$$
Q(X)=-a x_1^2 -p x_2^2 + ap x_3^2 =\epsilon_1,
$$
then $(-\epsilon_1, p)_p=(a, p)_p=-1$. In this case, we choose $X_1
=x_1 e$ such that $Q(X_1) =-a x_1^2 =\epsilon_1$. Since $(\mathbb
Z_p X_1)^\perp = \mathbb Z_p \pi + \mathbb Z_p \pi e$,  finding
$T(\vec x) =T$ with $\vec x={}^t(X_1, X_2)$ is the same as finding
$X_2 =y_2 \pi + y_3 \pi e$ with
$$
Q(X_2) =-p y_2^2 + p a y_3^2 = \epsilon_2 p^t,
$$
that is
$$
y_2^2 -a y_3^2 = -\epsilon_2 p^{t-1}.
$$
Since $(a, p)_p=(-\epsilon_1, p)_p=-1$ and $(a,
-\epsilon_2)_p=(-\epsilon_1, -\epsilon_2)_p=1$, it is equivalent to
$t-1$ being even. So there is $\vec x \in L_p^2$ such that $T(\vec
x) =T$ if and only if
$$
(-\epsilon_1, p)_p^t =-1.
$$
Assuming this condition, choose one $\vec x_0 \in L_p^2 $ with
$T(\vec x_0) =T$. Notice that
$$
\mathbb Q_p^* \backslash \mathbb B_p^* = \O_p^* \cup \pi \O_p^*
$$
and $\pi.L_p^2 =L_p^2$. So in this case,
$$
\beta_p(\diag(\epsilon_1, \epsilon_2 p^t, \Psi_p) = \int_{\mathbb
Q_p^* \backslash \mathbb B_p^*/\O_p^*} \Psi_p(h^{-1}.\vec x_0)  dh
=2.
$$
In summary, we have
$$
\beta_p(T_q(\mu n), \Psi_p) = 1 -(-\alpha_p, p)_p^{t_p}.
$$
Now we assume $p=2$. In this case,
$$
\O_2=\mathbb Z_2 + \mathbb Z_2 i +\mathbb Z_2j +\mathbb Z_2
\frac{1+i+j+k}2,   \quad i^2 =j^2 =k^2 =-1, ij=-ji =k,
$$
and so
$$
L_2 =(\mathbb Z_2 + 2 \O_2 )\cap V_p = \mathbb Z_2 2i +\mathbb Z_2
2j + \mathbb Z_2 (i+j+k)
$$
is isomorphic to $\tilde L=\mathbb Z_2^3$ with quadratic form
\begin{equation}
Q(x, y, z) =3 x^2 + 8 (y^2+yz +z^2).
\end{equation}
In order for it to  represent $T=\diag(\epsilon_1, \epsilon_2 2^t)$
with $\epsilon_i \in \mathbb Z_2^*$ and $ t\in \mathbb Z_{\ge 0}$,
one has to have
$$
\epsilon_1 =3 x^2 + 8 (y^2+yz +z^2) \equiv 3 \mod 8.
$$
In such a  case, we may choose $x_0\in \mathbb Z_2^*$ such that
$x_0^2 = \epsilon_1/3$. Let $e=(x_0, 0, 0)\in \tilde L$, then
$Q(e)=\epsilon_1$. It is easy to see that $\tilde L$ represents $T$
if and only if $e^\perp$ represents $\epsilon_2 2^t$, i.e., $y^2 +
yz +z^2$ represents $\epsilon_2 2^{t-3}$, which is equivalent to
that $t-3 \ge 0$ is even. Now the argument as above gives that
$$
\beta_2(\diag(\epsilon_1, \epsilon_2 2^t), \Psi_2) = \begin{cases}
  2 &\ff \epsilon_1 \equiv 3 \mod 8, t\ge 3 \hbox{ odd},
  \\
   0 &\hbox{otherwise}.
   \end{cases}
   $$
For  $T_q(\mu n) =\diag(\alpha_2, \alpha_2^{-1} \det T_q(\mu n))$
one has $\epsilon_1 =\alpha_2 \equiv 3 \mod 4$ and $t =t_2 +2=\ord_2
\det T_q(\mu n) =\ord_2 \frac{q^2 \tilde D-n^2}{q^2 D} \ge 3$ since
$\frac{q^2 \tilde D-n^2}{q^2 D}  \in 8\mathbb Z_2$. So we still have
$$
\beta_2(T_q(\mu n), \Psi_2) = 1 - (-\alpha_2, 2)_2^{t_2}.
$$\end{proof}

\subsection{The case $l=q $}

Now  we come to the tricky case $l=q$.
 Recall $$ L_q' = \{ X=
\kzxz {x} {\frac{1}q y} {z} {-x} :\, x, y, z \in \mathbb Z_q \}.
$$
Let
\begin{equation}
\Omega_q'= \{ x={}^t(X_1, X_2) \in (L_q')^2:\, z_1 + z_2 \sqrt D
\equiv y_1 -y_2 \sqrt D \equiv 0 \mod q\}
\end{equation}
and $\Psi_q' =\cha \Omega_q'$. Let
$$
T_q'(\mu n) = \kzxz {1} {\frac{D}2} {0} {\frac{1}2}  T_q(\mu n)
\kzxz {1} {0} {\frac{D}2} {\frac{1}2} =\kzxz {a} {b} {b} {c}.
$$
Then
\begin{align} \label{oldeq5.16}
ac- b^2 &= \det T_q'(\mu n) = \frac{q^2 \tilde D -n^2}{4D q^2},
\notag
\\
\Delta &= -(a+Dc) -2b \sqrt D,
\\
 a- Dc &= -\mu \frac{n}D. \notag
\end{align}
Lemma \ref{oldlem5.1} implies that
\begin{equation} \label{eq5.14}
\beta_q(T_q(\mu n), \Psi_q) = \beta_q(T_q'(\mu n), \Psi_q').
\end{equation}

\begin{prop}  \label{oldprop5.5} When $q \nmid n$, one has
$$
\beta_q(T_q(\mu n), \Psi_q) =1.
$$
\end{prop}
\begin{proof} When $q \nmid n$, (\ref{oldeq5.16}) implies that $a ,
c \in \frac{1}q \mathbb Z_q^*$, and so
$$
T_q'(\mu n) =\kzxz {1} {0} {a^{-1} b} {1} \kzxz {a} {0} {0} {\tilde
a} \kzxz {1} {a^{-1} b} {0} {1}
$$
with $\tilde a= \frac{1}{q} \det T_q'(\mu n) \in \frac{1}q \mathbb
Z_q^*$. Since $b \in \mathbb Z_q$, $\kzxz {1} {0} {a^{-1} b} {1} \in
K_0(q)$,  and $\Psi_q'$ is $K_0(q)$-invariant (with respect to the
linear action), one has
$$
\beta_q(T_q'(\mu n), \Psi_q') =\beta_q(\diag(a, \tilde a), \Psi_q').
$$
Since
$$-\frac{\tilde a}{a} = -\frac{\det T_q'(\mu n)}{a^2} = \frac{n^2
-q^2 \tilde D}{4 D (qa)^2}\equiv \frac{n^2}{4D (qa)^2} \mod q
$$
there is $z_0 \in \mathbb Z_q^*$ with $z_0^2 =-\frac{\tilde a}a$.
Set $\vec x_0 = {}^t(X_1, X_2) \in (L_q')^2$ with
$$
X_1= \kzxz {0} {-a} {1} {0}, \quad  X_2 = \kzxz {0} {a z_0} {z_0}
{0}.
$$
Then $T(\vec x_0) = \diag(a, \tilde a)$. It is easy to check that
$h(r, u)^{-1}. \vec x_0 \in (L_q')^2$ if and only if $r=0$ and $u\in
\mathbb Z_q$, i.e., $h(r, u) =1 \mod K_0(q)$. In this case, $\vec
x_0 \in \Omega_q'$ if and only if $1 +z_0 \sqrt D = 0\mod q$.

On the other hand, it is easy to check $h'(r, u)^{-1}.\vec x_0 \in
(L_q')^2$ if and only if $r=-1$ and $u \in \mathbb Z_q$, i.e., $h(r,
u) = \kzxz {0} {1} {1} {0} \mod K_0(q)$. In this case, $h'(-1,
0)^{-1}.\vec x_0 \in \Omega_q'$ if and only if $1 -z_0 \sqrt D
\equiv 0\mod q$.

Since
$$
1- z_0^2 D = 1 + \frac{\tilde a}{a} D = \frac{ q (qa) (a+Dc) - q^2
b^2}{(qa)^2} \equiv 0 \mod q,
$$
exactly one of the following holds: $1 +z_0 \sqrt D = 0\mod q$ or $1
-z_0 \sqrt D \equiv 0\mod q$. So there is exactly  one coset
$\mathbb Q_q^* h K_0(q)$ such that $h^{-1}.\vec x_0 \in \Omega_q'$.
This proves $\beta_q(\diag(a, \tilde a), \Psi_q') =1$, and  thus the
lemma.
\end{proof}

Next, we assume that $q|n$. In this case $T_q'(\mu n) \in
\Sym_2(\mathbb Z_q)$. Actually, $T_q(\mu n) = T(\mu \frac{n}q)$ in
the notation of \cite{m=1}. So there is $g =\kzxz {g_1} {g_2} {g_3}
{g_4} \in \SL_2(\mathbb Z_q)$ such that
\begin{equation} \label{oldeq5.18}
T_q'(\mu n) = g T {}^tg, \quad T=\diag(\epsilon_1, \epsilon_2 q^t)
\end{equation}
with $\epsilon_i \in \mathbb Z_q^*$, and $t =\ord_{q} \det T_q'(\mu
n) =\ord_q \frac{q^2 \tilde D-n^2}{4D q^2}$.

For $v_1, v_2 \in \mathbb Z/q\mathbb Z$, we set
\begin{align}
\Omega_{v_1, v_2} &=\{ \vec x={}^t(X_1, X_2) \in L_q^2:\, v_1 z_1 +
v_2
z_2 = 0 \mod q \} \\
&=\{ \vec x={}^t(X_1, X_2) \in L_q^2:\, v_1 X_1 + v_2 X_2 \in L_0(q)
\} \notag
\end{align}
where $L_q =M_2(\mathbb Z_q)$ and
\begin{equation}
L_0(q) =\{ X=\{ \kzxz {x} {y} {q z} {-x} \in V_q:\, x, y, z \in
\mathbb Z_q \}.
\end{equation}
Let
\begin{equation}
\Psi_{v_1, v_2} =\cha (\Omega_{v_1, v_2}), \quad \Psi_0 =\cha
(L_0(q)^2).
\end{equation}

\begin{lem} \label{oldlem5.6}  Let $T_q'(\mu n) =g T {}^tg$ be as in (\ref{oldeq5.18}), and let
$$
\kzxz {v_1} {v_2} {v_3} {v_4} = \kzxz {g_1+g_3 \sqrt D} {g_2 +g_4
\sqrt D} {g_1 -g_3 \sqrt D} {g_2 -g_4 \sqrt D} = \kzxz {1} {\sqrt D}
{1} {-\sqrt D} \kzxz {g_1} {g_2} {g_3} {g_4}.
$$
Then
$$
\beta_q(T_q'(\mu n), \Psi_q')
 = \beta_q(T, \Psi_{v_1, v_2}) + \beta_q(T, \Psi_{v_3, -v_4}) -
 \beta_q(T, \Psi_0).
 $$
\end{lem}
\begin{proof}  Lemma \ref{oldlem5.1} implies that
$$
\beta_q(T_q'(\mu n), \Psi_q') = \beta_q(T, f)
$$
with $f(\vec x) = \Psi_q'(g \vec x)$. So $f(\vec x) \ne 0$ if and
only if $g\vec x \in \Omega_q'$, i.e., $\vec x= {}^t(X_1, X_2) \in
(L_q')^2$ with $X_i =\kzxz {x_i} {\frac{1}q y_i} {z_i} {-x_i}$ and
\begin{align}
 v_1 z_1 + v_2 z_2 &\equiv 0 \mod q, \label{oldeq5.22}
 \\
 v_3 y_1 -v_3 y_2 &\equiv 0 \mod q. \label{oldeq5.23}
 \end{align}
Since $T \in \Sym_2(\mathbb Z_q)$, to have $T(\vec x) =T$ for $\vec
x \in (L_q')^2$, one has to have
$$
y_1 z_1, y_2 z_2, y_1 z_2 +y_2 z_1 \in q \mathbb Z_q
$$
and so  either $ y_1, y_2 \equiv  0 \mod q$, i.e., $\vec x \in
L_q^2$, or $z_1, z_2 \equiv 0 \mod q$, i.e., $\kzxz {0} {q^{-1}} {1}
{0} .\vec x \in L_q^2$.

When $y_1, y_2 \equiv 0 \mod q$, (\ref{oldeq5.22}) is automatic and
thus $g \vec x \in \Omega_q'$ if $\vec x \in \Omega_{v_1, v_2}$.
When $z_1, z_2 \equiv 0 \mod q$, $(\ref{oldeq5.23})$ is automatic,
and $g\vec x \in \Omega_q'$ if and only if $\vec x \in \kzxz {0} {1}
{q} {0}.\Omega_{v_3, -v_4}$. When $y_1, y_2, z_1, z_2 \equiv 0 \mod
4$, $g\vec x \in \Omega_q'$ automatically and $\vec x \in L_0(q)^2$.
So we have
\begin{align*}
\beta_q(T_q'(\mu n), \Psi_q')
  &= \beta_q(T, \Psi_{v_1, v_2})
    + \beta_q(T, \cha\left( \kzxz {0} {1} {q} {0}.\Omega_{v_3,-v_4}\right))
    -\beta_q(T, \Psi_0)
    \\
     &=\beta_q(T, \Psi_{v_1, v_2})
    + \beta_q(T, \Psi_{v_3, -v_4})
    -\beta_q(T, \Psi_0)
    \end{align*}
    as claimed.
\end{proof}

As in Section 5.2, there exists $\vec x={}^t(X_1, X_2) \in V_q^2$
with $T(\vec x) = T$ if and only if $(-\epsilon_1, q)_q^t=1$. Choose
$\vec x_0 ={}^t(X_1, X_2)$ as in (\ref{oldeq5.11}) (with $l$
replaced by $q$). The following lemma is contained in the proof of
Proposition \ref{oldprop5.3}.

\begin{lem} \label{oldlem5.7} (1) \quad  When $(-\epsilon_1, q)_q=-1$ and $t$ is even,
$$ h(r, u)^{-1}.\vec x_0 \in L_q^2 \Leftrightarrow h'(r,
u)^{-1}.\vec x_0 \in L_q^2 \Leftrightarrow r=0, u \in \mathbb Z_q.
$$

(2) \quad When $(-\epsilon_1, q)_q=1$,
$$
h(r, u)^{-1}.\vec x_0 \in L_q^2 \Leftrightarrow
 h'(r, u)^{-1}.\vec
x_0 \in L_q^2 \Leftrightarrow 0\le r\le t, u =-\frac{y_0}{ x_0} \mod
q^r.
$$
\end{lem}

We first consider a special case $t=0$ which is different from the
case $t>0$.

\begin{lem} \label{oldlem5.8} Let $v_1, v_2 \in \mathbb Z/q$ with at least one being nonzero. One has
$$
\beta_q(\diag(\epsilon_1, \epsilon_2), \Psi_{v_1, v_2})
=\begin{cases} 2 &\ff  -(\epsilon_1 v_1^2 + \epsilon_2 v_2^2) \equiv
\square \mod q,
\\
 0 &\hbox{otherwise}.
 \end{cases}
 $$
 \end{lem}
\begin{proof}
By the above lemma,  we only need to check whether $\vec x_0$ and
$h'(0, u).\vec x_0$ belong to $\Omega_{v_1, v_2}$ with $u \in
\mathbb Z/q$. $\vec x_0 \in \Omega_{v_1, v_2}$ if and only if $v_1
- v_2 y_0 \equiv 0 \mod q$. Since
\begin{align*}
h'(0,u)^{-1}.X_1 &= \begin{pmatrix}  {-\epsilon_1 u} &{-\epsilon_1}
\\{1+\epsilon_1 u^2} &{\epsilon_1 u} \end{pmatrix} ,
\\
 h'(0, u)^{-1}.X_2&= \begin{pmatrix} -x_0 +\epsilon_1 y_0 u &
 \epsilon_1 y_0 \\ y_0+ 2 x_0 u -\epsilon_1 y_0 u^2 & x_0
 -\epsilon_1 y_0 u \end{pmatrix},
 \end{align*}
$h'(0, u)^{-1}.\vec x_0 \in \Omega_{v_1, v_2}$ if and only if
\begin{equation}\label{oldeq5.24}
\epsilon_1( v_1 -v_2 y_0) u^2 + 2 x_0 v_2 u + (v_1 + v_2 y_0) \equiv
0 \mod q.
\end{equation}
When $v_1 -v_2 y_0 \equiv 0 \mod q$, $v_2 \not\equiv 0 \mod q$, and
thus (\ref{oldeq5.24}) has one solution $\mod q$. When $v_1 -v_2 y_0
\not\equiv 0 \mod q$,  (\ref{oldeq5.24}) has either two or zero
solutions mod $q$ depending on whether its discriminant
$$
(2x_0 v_2)^2 - 4 \epsilon_0 (v_1-v_2y_0)(v_1 +v_2 y_0)
=-4(\epsilon_1 v_1^2 +\epsilon_2 v_2^2)
$$
is a square or not mod $q$ (recall $x_0^2 + \epsilon_1 y_0^2
=-\epsilon q^t$). Notice that when $v_1 -v_2 y_0 \equiv 0 \mod q$,
$-(\epsilon_1 v_1^2 + \epsilon_2 v_2^2) = x_0^2 v_2^2$ is a square.
This proves the lemma.
\end{proof}

\begin{prop} \label{oldprop5.9} When $q|n$ and $\det T_q(\mu) =\frac{q^2 \tilde
D-n^2}{ D q^2} \in \mathbb Z_q^*$, one has
$$
\beta_q(T_q(\mu n), \Psi_q) =\begin{cases}
  4 &\ff q \hbox{ split completely in } \tilde K,
  \\
   2 &\ff q \hbox{ inert in } \tilde F, q\O_{\tilde F} \hbox{ split
    in } \tilde K,
    \\
    0  &\hbox{otherwise}.
    \end{cases}
    $$
    \end{prop}
\begin{proof} Write $T_q'(\mu n) = g T {}^tg $ with $g \in
\GL_2(\mathbb Z_q)$ and $T =\diag(1, \epsilon)$, $\epsilon= \det
T_q'(\mu n) = \frac{q^2 \tilde D-n^2}{4 D q^2} \in \mathbb Z_q^*$ as
above. Then
$$
g_1^2 + g_2^2 \epsilon =a, \quad g_1 g_3 + g_2 g_4 \epsilon =b,
\quad g_3^2 + g_4^2 \epsilon =c.
$$
So Lemmas \ref{oldlem5.6} and (\ref{oldeq5.16}) imply
$$
v_1^2 + \epsilon v_2^2=(g_1 +g_3\sqrt D)^2 + \epsilon (g_2 +g_4\sqrt
D)^2 =a +Dc +2b \sqrt D = -\Delta
$$
and
$$
v_3^2 + \epsilon v_4^2 =-\Delta'.
$$
Now applying Lemma \ref{oldlem5.8}, one obtains
$$
\beta_q(T, \Psi_{v_1, v_2}) = \begin{cases}
  2 &\ff \Delta  \in (\mathbb Z_q^*)^2,
  \\
   0  &\ff \Delta  \notin (\mathbb Z_q^*)^2,
   \end{cases}
   $$
   and
$$
\beta_q(T, \Psi_{v_3, -v_4}) = \begin{cases}
  2 &\ff \Delta'  \in (\mathbb Z_q^*)^2,
  \\
   0  &\ff \Delta'  \notin (\mathbb Z_q^*)^2,
   \end{cases}
   $$
Since $\epsilon_i \in \mathbb Z_q^*$, it is easy to see that
$\beta_q(T, \Psi_0)=0$. So  Lemma \ref{oldlem5.6}  and
(\ref{eq5.14}) imply
$$
\beta_q(T_q(\mu n), \Psi_q)=\begin{cases}
 4 &\ff \Delta, \Delta' \in (\mathbb Z_q^*)^2,
 \\
  2 &\ff \hbox{ exactly one of } \Delta \hbox{ or } \Delta' \in (\mathbb Z_q^*)^2,
  \\
  0 &\hbox{otherwise}.
  \end{cases}
$$
Recall that $q=\mathfrak q \mathfrak q'$ is split in $F$, and under
the identification $F \hookrightarrow F_{\mathfrak q} \cong \mathbb
Q_q$, $\sqrt D$ goes to $\sqrt D$. So $\Delta \in (\mathbb Z_q^*)^2$
if and only if $\mathfrak q$ splits in $K$. $\Delta' \in (\mathbb
Z_q^*)^2$ if and only if $\mathfrak q'$ splits in $K$.

Consider the diagram of fields:

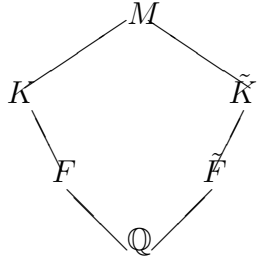
\setlength{\unitlength}{1mm}
\begin{center}
\begin{picture}(30, 30)(-15, -15)
\put(0,-15){${\mathbb Q}$\line(1,1){8}}
 \put(0,-15){\line(-1,1){8}}

\put(-10, -6){{$F$}}
\put(-10, -4.5){   \line(-1,2){4}}

\put(10, -6){{$\tilde F$}}
 \put(11.5, -4.5){\line(1,2){4}}

 \put(-16,4.5){{$K$}}
\put(-15,6.5){ \line(3,2){13.5}}



 \put(13.5,4.5){$\tilde K$}
\put(15,6.5) { \line(-3,2){13.5}}

\put(0, 15){$M$}

\end{picture}
\end{center}

When $q=\tilde{\mathfrak q}\tilde{\mathfrak q}'$ is split in $\tilde
F$, $(\Delta \Delta', q)_q = (\tilde D, q)_q=1$. So either $q$
splits completely in $K$ and thus in $M=K \tilde K$ or both
$\mathfrak q$ and $\mathfrak q'$ are inert in $K$. Similarly, since
$q$ is split in $F$, either $q$ splits  completely in $\tilde K$ and
thus in $M$ or both $\tilde{\mathfrak q}$ and $\tilde{\mathfrak q}'$
are inert in $\tilde K$. Therefore, under the condition that $q$ is
split in $\tilde F$, we have
$$
\beta_q(T_q(\mu n), \Psi_q)=\begin{cases}
 4 &\ff q \hbox{ split completely in } \tilde K,
 \\
 0 &\hbox{otherwise}.
 \end{cases}
 $$

 When $q$ is inert in $\tilde F$, $(\Delta \Delta', q)_q=(\tilde D,
 q)_q=-1$, exactly one of $\Delta$ or $\Delta'$ is a square in
 $\mathbb Z_q^*$. This implies that there are at least three primes
 of $M$ above $q$, and thus that $q\O_{\tilde F}$ has to be split in
 $\tilde K$. This finishes the proof of the proposition.
\end{proof}

Finally we consider the case $t \ge 1$ and prove

\begin{prop} \label{oldprop5.10} Assume that $q|n$ and $t_q= \ord_{q} \frac{q^2\tilde D- n^2}{4D q^2} >0$, and let $T_q(\mu n)$
is $\mathbb Z_q$-equivalent to $\diag(\alpha_q, \alpha_q^{-1} \det
T_q(\mu n))$ with $\alpha_q \in \mathbb Z_q^*$. Then
$$
\beta_q(T_q(\mu  n),\Psi_q) =\begin{cases}
  0 &\ff (-\alpha_q, q)_q=-1,
  \\
  2(t_q+2) &\ff (-\alpha_q, q)_q =1.
  \end{cases}
  $$
  \end{prop}
  \begin{proof} Since $T_q'(\mu n)$ is $\mathbb Z_q$-equivalent to
  $T_q(\mu n)$, it is also $\mathbb Z_q$ equivalent to $\diag(\alpha_q, \alpha_q^{-1} \det
T_q(\mu n))$, which we now shorten as $T=\diag(\epsilon_1, \epsilon_2
q^t)$ with $\epsilon_1 =\alpha_q, \epsilon_2 \in \mathbb Z_q^*$ and
$t=t_q$. As in the proof of Proposition \ref{oldprop5.9}, we write
$T_q'(\mu n) =g T {}^tg$ so that
$$
\beta_q(T_q'(\mu n), \Psi_q') =\beta_q(T, \Psi_{v_1, v_2})
+\beta_q(T, \Psi_{v_3, -v_4}) -\beta_q(T, \Psi_0).
$$
Here $v_i$ are given as in Lemma \ref{oldlem5.7}.

{\bf Case 1:} We  first assume that $(-\epsilon_1, q)_q =-1$, so
$t=2 t_0$ is even. In this case $x_0, y_0 \in q^{t_0} \mathbb Z_q$
and thus $x_0, y_0 \equiv 0 \mod q$.  In order to compute
$\beta_q(T, \Psi_{v_1, v_2})$, we only need to consider whether
$\vec x_0$ and $h'(0, u)^{-1}.\vec x_0$ belong to $\Omega_{v_1,
v_2}$ by Lemma \ref{oldlem5.6}. It is easy to check as before that
$\vec x_0 \in \Omega_{v_1, v_2}$ if and only if $v_1 -v_2 y_0 \equiv
v_1 \equiv 0\mod q$, and $h'(0, u)^{-1}.\vec x_0 \in \Omega_{v_1,
v_2}$ if and only if
$$
v_1(1+\epsilon_1 u^2) + v_2 (y_0 +x_0 u-\epsilon_1 y_0 u^2) \equiv 0
\mod q
$$
i.e., $v_1 \equiv 0 \mod q$. On the other hand, the same calculation
as in the proof of Proposition \ref{oldprop5.9} gives
$$
-(\epsilon_1 v_1^2 + \epsilon_2 q^t v_2^2) =\Delta \ne 0 \mod q
$$
and thus $v_1 \not\equiv 0 \mod q$. So $\beta_q(T, \Psi_{v_1,
v_2})=0$. For the same reason, $\beta_q(T, \Psi_{v_3, -v_4}) =0$,
and thus $\beta_q(T_q(\mu n), \Psi_q)=\beta_q(T_q'(\mu n), \Psi_q')
=0$.

{\bf Case 2:}  Now we assume $(-\epsilon_1, q)_q=1$. By Lemma
\ref{oldlem5.6}, we need to consider how many  $h(r, u)^{-1}.\vec
x_0$ and $h'(r, u)^{-1}.\vec x_0$ are in $\Omega_{v_1, v_2}$, with
$0\le r \le t$ and $u \equiv -\frac{y_0}{x_0} \mod q^r$. In the case
$h(r, u)$ we count the number of $u\mod q^r$ classes, and in the
case $h'(r, u)$ we count  the number of $u \mod q^{r+1}$ classes.

When $r=0$, the same argument as in the proof of Proposition
\ref{oldprop5.9} shows that there are two classes of $h$ among $h(0,
u)$ and $h'(0, u)$ satisfying $h^{-1}.\vec x_0 \in L_{v_1, v_2}$,
since $q$ splits completely in $\tilde K$. Indeed, let $n_1 =n/q \in
\mathbb Z$. Then $t=t_q>0$ means $q|\frac{\tilde D- n_1^2}{4D}$ and
thus $q$ splits in $\tilde F$. Now \cite[Lemma 6.2]{m=1} implies
that one prime of $\tilde F$ above $q$ splits in $\tilde K$. Since
$q$ is split in $F$, this implies that both primes of $\tilde F$
above $q$ split in $ \tilde K$, i.e., $q$ splits completely in
$\tilde K$.

When $r >0$, $h(r, u)^{-1}.\vec x_0 \in \Omega_{v_1, v_2}$
automatically. On the other hand, the same calculation as in the
proof of Lemma \ref{oldlem5.8} shows that  $h(r, u)^{-1}.\vec x_0
\in \Omega_{v_1, v_2}$ if and only if
\begin{equation} \label{oldeq5.25}
\epsilon_1 (v_1 -v_2 y_0) u^2 + 2 x_0 v_2 u +(v_1 +v_2 y_0) \equiv 0
\mod q^{r+1}.
\end{equation}
Since $u \equiv - \frac{y_0}{x_0} \mod q^r$, we write $u
=-\frac{y_0}{x_0} + q^r \tilde u$. Now (\ref{oldeq5.25}) becomes
$$
\frac{2 \epsilon_1 y_0 v_1}{x_0} q^r \tilde u + (v_1 -v_2 y_0)
\frac{-\epsilon_2 q^t}{x_0^2} \equiv 0 \mod q^{r+1}.
$$
Since $\epsilon_1 v_1^2 + \epsilon_2 q^t \epsilon_2^2 =-\Delta
\not\equiv 0\mod q$, one has  $v_1 \not\equiv 0 \mod q$. So the
above equation has a unique solution $\tilde u \mod q$, and there is
a unique $u \mod q^{r+1}$ for $1 \le r \le t$ such that $h'(r,
u)^{-1}.\vec x_0 \in \Omega_{v_1, v_2}$. In summary, we have proved
$$
\beta_q(T, \Psi_{v_1, v_2}) =2t+2.
$$
For the same reason, $\beta_q(T, \Psi_{v_3, -v_4}) =2t+2$. A similar
argument gives $\beta_q(T, \Psi_0) =2t$. Therefore,
$$
\beta_q(T_q(\mu n), \Psi_q) =\beta_q(T, \Psi_{v_1, v_2}) +
\beta_q(T, \Psi_{v_3, -v_4}) -\beta_q(T, \Psi_0) =2t+4.
$$
\end{proof}

\section{ Computing $b_m(p)$ and Proof of Theorem \ref{theo1.5} }
\label{sect6}

In this section, we compute $b_m(p)$ assuming $(m, 2 D \tilde D p)
=1$ and prove the following theorem. A little more work could remove
the restriction. At the end of this section, we prove Theorem
\ref{theo1.5}, which is  clear after all these preparations.

\begin{theo} \label{theo6.1} Assume (\ref{eqOK}) and that $\tilde D =\Delta \Delta' \equiv 1 \mod 4$ is square free,  and that $m >0$ is square-free with $(m, 2D\tilde D p)=1$. Let $t_l=\ord_l \frac{m^2 \tilde D-n^2}{4 D m^2}$. Then
\begin{equation}\label{eq6.1}
b_m(p) = \sum_{ \substack{0< n <m \sqrt{\tilde D}\\ \frac{m^2
\tilde D -n^2}{4D} \in p \mathbb Z_{>0}}} (\ord_p \frac{m^2\tilde
D -n^2}{4D} +1)\sum_{\mu} b(p,\mu n, m)
\end{equation}
where
\begin{equation} \label{eq6.2}
b(p, \mu n, m) = \prod_{l |\frac{m^2 \tilde D-n^2}{4 D}} b_l(p,
\mu n, m)
\end{equation}
is given as follows.

(1) \quad When $l \nmid m$ and $l| \frac{m^2 \tilde D-n^2}{4 D}$,
$T_m(\mu n)$ is $\mathbb Z_l$-equivalent to
$\diag(\alpha_l,\alpha_l^{-1} \det T_m(\mu n))$ with $\alpha_l \in
\mathbb Z_l^*$, and
\begin{equation}
b_l(p, \mu n, m) =
  \begin{cases}
  \frac{1- (- \alpha_p, p)_p^{t_p}}2 &\ff l=p,
  \\
   t_l+1 &\ff l\nmid mp, (-\alpha_l,
   l)_l=1,
   \\
    \frac{1 + (-1)^{t_l}}2 &\ff l \nmid mp, (-\alpha_l,
   l)_l=-1.
   \end{cases}
 \end{equation}

(2) \quad When $l |m$, and $t_l=0$, one has
\begin{equation}
b_l(p, \mu n, m) =\begin{cases}
 4 &\ff  l \hbox{ split completely in } M,
 \\
 2 &\ff l \hbox{  inert in } \O_{\tilde F}, l\O_{\tilde F} \hbox{
 split in  } \tilde K,
 \\
  0 &\hbox{otherwise}.
  \end{cases}
  \end{equation}
Here $M=K \tilde K$ is the Galois closure of $K$ (and $\tilde K$)
over $\mathbb Q$.

  (3) \quad When $l|m$ is split in $F$ and $t_l >0$, $T_m(\mu n)$ is $\mathbb Z_l$-equivalent to
  $\diag(\alpha_l,\alpha_l^{-1} \det T_m(\mu n))$ with
$\alpha_l\in \mathbb Z_l^*$, and
  \begin{equation}
b_l(p, \mu n, m) =\begin{cases}
   0 &\ff (-\alpha_l, l)_l=-1,
   \\
   2(t_l+2) &\ff (-\alpha_l, l)_l=1.
   \end{cases}
   \end{equation}

(4) \quad When $l|m$ is inert in $F$ and $t_l>0$, $T_m(\mu n)$ is
$\mathbb Z_l$-equivalent to
  $\diag(\alpha_l,\alpha_l^{-1} \det T_m(\mu n))$ with
$\alpha_l\in \mathbb Z_l^*$, and
\begin{equation}
b_l(p, \mu n, m) =\begin{cases}
   1-(-1)^{t_l} &\ff (-\alpha_l, l)_l=-1,
   \\
   0 &\ff (-\alpha_l, l)_l=1.
   \end{cases}
   \end{equation}
\end{theo}
\begin{proof}
Recall
\begin{equation}
b_m(p)= \sum_{\mathfrak p|p} \sum_{\substack{ 0 < n < m \sqrt{\tilde D}\\
\frac{m^2 \tilde D -n^2}{4 D} \in p \mathbb Z_{>0}}} \sum_\mu \rho(t_n
d_{\tilde K/\tilde F} \mathfrak p^{-1}),
\end{equation}
with ($\mu = \pm 1$)
$$
t_{\mu n} = \frac{\mu n +m \sqrt{\tilde D}}{2D} \in d_{\tilde K/\tilde
F}^{-1}.
$$
 Clearly,  $b_m(p) =0$ unless
there is an integer $0 < n < m\sqrt{\tilde D}$ such that $\frac{m^2
\tilde D -n^2}{4D} \in p\mathbb Z_{>0}$.   Fix such an integer $n$
and recall $T_m(\mu n)$ from Lemma \ref{lemold1.1}.

  The condition $\frac{m^2 \tilde D -n^2}{4D} \in p \mathbb Z_{>0}$ implies that either $p$ is split in $\tilde F$ or
  $p|\hbox{gcd}(D, n)$ is ramified in $\tilde F$.
 In the ramified case, we have $p\O_{\tilde F}
  =\mathfrak p^2$. In the split case, we choose the splitting
  $p\O_{\tilde F} =\mathfrak p \mathfrak p'$ so that
  \begin{equation}
   t_{\mu n}=\frac{\mu n +m \sqrt{\tilde D}}{2D} \in \mathfrak p
   d_{\tilde K/\tilde F}^{-1}
   \end{equation}
  satisfies
\begin{equation}
\ord_\mathfrak p t_{\mu n} = \ord_p \frac{m^2\tilde D -n^2}{4D},
\quad \ord_{\mathfrak p'} (t_{\mu n}) =0  \hbox{ or } -1.
\end{equation}
With this notation, we have by definition
\begin{equation}
b_m(p) = \sum_{ \substack{ 0< n <m \sqrt{\tilde D}\\ \frac{m^2
\tilde D -n^2}{4D} \in p \mathbb Z_{>0}}} (\ord_p \frac{m^2\tilde
D -n^2}{4D} +1)\sum_{\mu} b(p,\mu n, m)
\end{equation}
where
\begin{equation}
b(p, \mu n, m) = \begin{cases}
  0 &\ff \mathfrak p \hbox { is  split in } \tilde K,
  \\
  \rho(t_{\mu n} d_{\tilde K/\tilde F} \mathfrak p^{-1}) &\ff \mathfrak p \hbox { is not  split in } \tilde
  K.
  \end{cases}
  \end{equation}

Assume now that $\mathfrak p$ is not split in $\tilde K$. Notice
that
$$
\rho(t_{\mu n}d_{\tilde K/\tilde F} \mathfrak
p^{-1})=\prod_{\mathfrak l} \rho_{\mathfrak l}(t_{\mu n}d_{\tilde
K/\tilde F} \mathfrak p^{-1})
$$
where the product runs over all  prime ideals  $\mathfrak l$ of
$\tilde F$, and
\begin{equation} \label{eqb.9}
\rho_{\mathfrak l}(t_{\mu n} d_{\tilde K/\tilde F}\mathfrak
p^{-1})
 = \begin{cases}
   1    &\ff    \mathfrak l \hbox{ is ramified in  } \tilde K,
   \\
   \frac{1 +(-1)^{\ord_{\mathfrak l} (t_{\mu n} d_{\tilde K/\tilde F} \mathfrak
   p^{-1})}}2 &\ff  \mathfrak l \hbox{ is inert in  } \tilde K,
   \\
   1+\ord_{\mathfrak l} (t_{\mu n} d_{\tilde K/\tilde F }\mathfrak
   p^{-1})
    &\ff  \mathfrak l \hbox{ is split  in  } \tilde K.
    \end{cases}
    \end{equation}
We write (assuming that $\mathfrak p$ is not split in $\tilde F$)
\begin{equation}
b(p, \mu n, m) =\prod_{l} {b_l(p, \mu n, m)}
\end{equation}
with
\begin{equation}
b_l(b, \mu n, m) = \prod_{\mathfrak l| l} \rho_{\mathfrak
l}(t_{\mu n} d_{\tilde K/\tilde F} \mathfrak p^{-1}).
\end{equation}
Clearly $b_l(b, \mu n , m) =1$ if $l \nmid \frac{m^2 \tilde D
-n^2}{4D p}$.  When $l |  \frac{m^2 \tilde D -n^2}{4D p}$, there
are three cases:

 (a) $l|m$,

  (b)  $l\nmid m$ and
$l|\hbox{gcd}(\tilde D, n)$ is ramified in $\tilde F$, or

 (c)
$l\nmid m$, and $l\O_{\tilde F} =\mathfrak l \mathfrak l'$ is
split in $\tilde F$.

 In case (c), we choose the ideal $\mathfrak
l$ so that
\begin{equation} \label{eqb.11}
\ord_{\mathfrak l} (t_{\mu n} d_{\tilde K/\tilde F}\mathfrak
p^{-1}) =\ord_l \frac{m^2\tilde D -n^2}{4 D p}=\ord_{l}\ord_l
\frac{m^2\tilde D -n^2}{4 D pm^2} , \quad \ord_{\mathfrak l'}
(t_{\mu n} d_{\tilde K/\tilde F}\mathfrak p^{-1})=0.
\end{equation}

Since $m$ does not affect local calculation in cases (b) and (c),
the same proof as in \cite[Lemma 6.2]{m=1} gives

\begin{lem}  \label{lemb.4} Let the notation be as above. Assume $l|\frac{m^2\tilde D -n^2}{4D}$, $l \nmid m$ and $\mathfrak l\ne d_{\tilde K/\tilde F}$.  Then
 $T_m(\mu n)$ is $\GL_2(\mathbb Z_l)$-equivalent to $\diag(\alpha_l,
 \alpha_l^{-1} T_m(\mu n))$ with $\alpha_l \in \mathbb Z_l^*$.
 Moreover,
 $\tilde K/\tilde F$ is split (inert) at $\mathfrak l$ if and only if $(-\alpha_l,
 l)_l=1$ (resp. $-1$).
 \end{lem}

\begin{prop}  \label{prop6.3} One has always
$$
b(p, \mu n, m) =\prod_{l|\frac{m^2 \tilde D -n^2}{4D}} b_l(p, \mu
n, m)
$$
with
$$
b_l(p, \mu n, m) =
  \begin{cases}
  \frac{1- (- \alpha_p, p)_p^{t_p}}2 &\ff l=p,
  \\
   t_l+1 &\ff l\nmid mp, (-\alpha_l,
   l)_l=1,
   \\
    \frac{1 + (-1)^{t_l}}2 &\ff l \nmid mp, (-\alpha_l,
   l)_l=-1.
   \end{cases}
   $$
   Here $T_m(\mu n)$ is $\GL_2(\mathbb Z_l)$-equivalent to $\diag(\alpha_l,
 \alpha_l^{-1} \det T_m(\mu n))$ with $\alpha_l \in \mathbb Z_l^*$, and
 $t_l=\ord_l T_m(\mu n) =\ord_l \frac{m^2 \tilde D -n^2}{4D m^2}$.
   \end{prop}
\begin{proof} First notice that the formula is true  even when
$\mathfrak p$ is split in $\tilde K$. Indeed,
$$
b_p(p, \mu n, m)=\frac{1- (-\alpha_p, p)_p^{t_p}}2 = 0
$$
since $(-\alpha_p, p)_p=1$ by Lemma \ref{lemb.4}. When $\mathfrak
p$ is non-split in $\tilde K$, the formulae follows from Lemma
\ref{lemb.4} and  (\ref{eqb.9})-(\ref{eqb.11}).

\end{proof}

{\bf Proof of Theorem \ref{theo6.1} (cont.)}:   Proposition
\ref{prop6.3}  settles Formulae (\ref{eq6.1}), (\ref{eq6.2}) and
Case (1) in the theorem. Now we assume $l|m$ and $l |\frac{m^2
\tilde D -n^2}{4D}$. This implies $l|n$. In this case we have
\begin{equation}
T_m(\mu n) =lT_{\frac{m}l}(\mu \frac{n}l).
\end{equation}
Write $m_1 =\frac{m}l$ and $n_1 =\frac{n}l$.

 (2) \quad Now we deal
with case (2): i.e., $l|m$ and $t_l=\frac{m^2 \tilde D -n^2}{4D
m^2}=\ord_l\frac{m_1^2 \tilde D -n_1^2}{4D} =0$.

{\bf Case 1:} \quad  If $l$ is inert in $\tilde F$, then $
\ord_{l} t_{\mu n} =1$. So
$$
b_l(p, \mu n, m) =\begin{cases}
  2 &\ff l\O_{\tilde F} \hbox{ is split in } \tilde K,
  \\
   0 &\ff l\O_{\tilde F} \hbox{ is inert in }\tilde K.
 \end{cases}
 $$

{\bf Case 2:} \quad If  $l=\mathfrak l \mathfrak l'$ is split in
$\tilde F$, then   $\ord_{\mathfrak l} t_{\mu n} =
\ord_{\mathfrak l'} t_{\mu n} =1$, and so
$$
b_l(p, \mu n, m) = \begin{cases}
  4 &\ff l \hbox{ split completely in } \tilde K,
  \\
  0 &\hbox{otherwise}.
  \end{cases}
  $$
On the other hand, $\Delta \Delta'=D v^2$ for some integer $v \ne
0$. So $l$ is split completely in $\tilde K$ implies that $(D,
l)_l=1$, i.e., $l$ is split in $F$ too, and thus $l$ is split
completely in $M$. This proves (2)

(3) \quad Now we assume $l|m$, $t_l>0$ and that $l$ is split in
$F$. in this case, $l|\frac{m_1^2 \tilde D -n_1^2}{4D}$. Since
$(m, 2 D \tilde D p)=1$, $l=\mathfrak l \mathfrak l'$ is split in
 $\tilde F$.
 Choose the splitting
in $\tilde F$  so that
\begin{equation} \label{eqb.13}
\ord_{\mathfrak l} t_{\mu n} = t_l +1, \quad \ord_{\mathfrak l'}
t_{\mu n} =1.
\end{equation}

 Since $l$ is split in $F$, $(D, l)_l=1$.  So $\tilde\Delta
\tilde{\Delta}' =D v^2$ implies that either both $\mathfrak l$ and
$\mathfrak l'$ are inert in $\tilde K$ or both are split  in
$\tilde K$. So
\begin{equation} b_l(p, \mu n, m)= \begin{cases}
  2 (t_l+2)  &\ff  l \hbox{  split completely in } \tilde K,
  \\
  0 &\hbox{otherwise}.
  \end{cases}
  \end{equation}
 Since $t_l >0$, applying Lemma
\ref{lemb.4} to the pair $(m_1, n_1)$, we see that $\tilde
K/\tilde F$ is split at $\mathfrak l$ if and only if $(-\alpha_l,
l)_l =1$. So we have
\begin{equation}
b_l(p, \mu n, m) =\begin{cases}
  0 &\ff (-\alpha_l , l)_l
  =-1,
  \\
  2 (t_l+2) &\ff (- \alpha_l, l)_l
  =1
  \end{cases}
  \end{equation}
as claimed.

(4) \quad Finally, we assume $l|m$, $t_l >0$, and $l$ is inert in
$F$. Just as in (3), $l=\mathfrak l\mathfrak l'$ is split in
$\tilde F$ and we can again
 choose the splitting as in (\ref{eqb.13}).
Since $(D, l)_l=-1$,  $\tilde\Delta \tilde{\Delta}' =D v^2$
implies that exactly one of $\mathfrak l$ and $\mathfrak l'$ is
split in $\tilde K$, and the other one is inert in $\tilde K$. So
\begin{equation}
b_l(p, \mu n, m) =\begin{cases}
   0 &\ff  \mathfrak l \hbox{ is split in } \tilde K,
   \\
   1 -(-1)^{t_l} &\ff \mathfrak l \hbox{ is inert in } \tilde K.
\end{cases}
\end{equation}
  Applying Lemma
\ref{lemb.4} to $(m_1, n_1)$ again, we obtain (4). This finishes
the proof of Theorem \ref{theo6.1}.
\end{proof}

{\bf Proof of Theorem \ref{theo1.5}:} By  Theorems \ref{theo3.5} and
\ref{theo4.1}, one has for $p \ne q$
$$
(\mathcal T_q.\CM(K))_p = \frac{1}2 \sum_{\substack{ 0 < n < q
\sqrt{\tilde D} \\ \frac{q^2 \tilde D-n^2}{4D} \in p\mathbb Z_{>0}}}
\left( \ord_p \frac{q^2 \tilde D -n^2}{4D} +1\right) \sum_\mu
\beta(p, \mu n)
$$
where
$$
\beta(p, \mu n) = \frac{1}2 \prod_{l} \beta_l(T_q(\mu n), \Psi_l)
$$
is computed in Section \ref{sect5}. By Theorems \ref{theo5.1} and
\ref{theo5.2}, one has $\beta_l(T_q(\mu n), \Psi_l)=1$ for $l
\nmid \frac{q^2 \tilde D -n^2}{4D}$, and so
$$
\beta(p, \mu n) = \frac{1}2 \prod_{l|\frac{q^2 \tilde D -n^2}{4D}
} \beta_l(T_q(\mu n), \Psi_l)
$$
Now comparing Theorems \ref{theo5.1} and \ref{theo5.2} with
Theorem \ref{theo6.1},  one sees that for $l| \frac{q^2 \tilde D-
n^2}{4D}$ (recall $q$ is a prime split in $F$)
$$
\beta_l(T_q(\mu n), \Psi_l)
 =\begin{cases} 2 b_p(p, \mu n, q)  &\ff l=p,
 \\
b_l(p, \mu n, q)  &\ff l\ne p.
\end{cases}
$$
and thus
$$
\beta(p, \mu n) = b(p, \mu n, q).
$$
Now applying Theorem \ref{theo6.1}, one sees
$$
(\mathcal T_q.\CM(K))_p =\frac{1}2 b_q(p)
$$
as claimed in Theorem \ref{theo1.5}.

\section{Faltings height and Proofs of Theorems \ref{maintheo}, \ref{Colmez} and \ref{theo1.4} }
\label{sect7}

\newcommand{\hPic}{\widehat{\operatorname{Pic}}(\tilde{\mathcal M}, \mathcal
D_{\hbox{pre}})}

\newcommand{\hCH}{\widehat{\operatorname{CH}}^1(\tilde{\mathcal M}, \mathcal
D_{\operatorname{pre}})}

\newcommand{\hZ}{\widehat{\operatorname{Z}}^1(\tilde{\mathcal M},
\mathcal D_{\operatorname{pre}})}

Let $\tilde{\mathcal M}$ be a toroidal compactification of $\mathcal
M$ and let $C= {\tilde{\mathcal M} -\mathcal M}$ be the boundary. We
need the Faltings height pairing in a slightly more general setting
as written in literature, i.e., on DM-stacks where Green functions
have pre-log-log growth along the boundary $C$ in the sense of
\cite{BKK}. We restrict to our special case to avoid introducing
more complicated concept `pre-log-log Green object', and refer to
\cite{BKK}  for detailed study in this subject, and to \cite[Section
1]{BBK} for a brief summary.

Let $N \ge 3$, and let $X$ be the moduli scheme over $\mathbb C$ of
abelian surfaces with real multiplication by $\O_F$ and with full
$N$-level structure \cite{Pa}, and let $\tilde X$ be a toroidal
compactification of $X$. Then  $M=\mathcal M(\mathbb C) = [\Gamma
\backslash X]$ and $\tilde M=\tilde{\mathcal M}(\mathbb C)=[\Gamma
\backslash \tilde X]$  are quotient stacks, where $\Gamma =\Gamma(N)
\backslash \SL_2(\O_F)$. Let $\pi$ be the natural map from $\tilde
X$ to $\tilde M$. Let $Z$ be a divisor of $\tilde M$, and let $Z_N
=\pi^{-1}(Z)$ be its preimage in $\tilde X$. Following  \cite[Chapter
2]{KRY2}, the Dirac current $\delta_Z$ on $\tilde M$ is given by
$$
\langle \delta_Z, f \rangle_{\tilde M }= \frac{1}{\# \Gamma} \langle
\delta_{Z_N}, f\rangle_{\tilde X}
$$
for every $C^\infty$ function on $\tilde M$ with compact support, which
is defined as a $\Gamma$-invariant $C^\infty$ functions on $X$
with compact support.  A  pre-log-log Green function for $Z$ is
defined to be a $\Gamma$-invariant pre-log-log Green function $g$
for $Z_N$, i.e., $g$  is $\Gamma$-invariant, has  log singularity
along $Z_N$ and  pre-log-log growth along $C$ in the sense of
\cite{BKK}, see also \cite[Section 1]{BBK} such that
$$
d d^c g + \delta_{Z_N} = [\omega]_{\tilde X}
$$
as currents for a $\Gamma$-invariant $C^\infty$ (log-log growth
along with $C$ and $C^\infty$ everywhere else) $(1, 1)$-form
$\omega$. When viewed as currents on $\tilde M$, one has also
$$
d d^c g + \delta_Z= [\omega]_{\tilde M}.
$$
Let  $\hZ$ be the abelian group of the pairs $(\mathcal Z, g)$ where
$\mathcal Z$ is a divisor of $\tilde{\mathcal M}$ and $g$ is a
pre-log-log Green function for $Z=\mathcal Z(\mathbb C)$. For a
rational function $f$ on $\mathcal M$,
$$
\widehat{\div} (f) = (\div f, - \log |f|^2) \in \hZ
$$
and let $\hCH$ be the quotient group of $\hZ$ by the subgroup
generated by all $\widehat{\div}(f)$. Let $\mathcal Z$ be a prime
 cycle  in $\mathcal M$ (not intersecting with the boundary $C$)
 of dimension $1$, and let $j: \mathcal Z \rightarrow
 \tilde{\mathcal M}$ be the natural embedding. Then $j$ induces a
 natural map
 \begin{equation}
 j^*:  \hCH_\mathbb Q \rightarrow \widehat{\CH}^1(\mathcal Z)_\mathbb Q,
 \end{equation}
which is given by
$$
j^*(\mathcal T, g) = (j^*\mathcal T, j^*g),  \quad  j^*(g) (z) =
g(j(z))
$$
when $\mathcal T$ and $\mathcal Z$ intersect properly. Here for an
abelian group $A$, we write $A_\mathbb Q$ for the $\mathbb Q$-vector
space $A\otimes \mathbb Q$. Since $\mathcal Z(\mathbb C)$ does not
intersect with the boundary $C$, $j^*g$ well-defined over $\mathcal
Z(\mathbb C)$. Here arithmetic Chow group $\widehat{\CH}^1(\mathcal
Z)$ is defined the same way as above except that the Green function
$g$ is $C^\infty$ (actually in special case, just constants at
points of $\mathcal Z(\mathbb C)$). In \cite[Chapter 2]{KRY2}, it is
shown that there is a linear map---the arithmetic degree
\begin{equation}
\widehat{\deg}: \widehat{\CH}^1(\mathcal Z)_\mathbb Q \rightarrow
\mathbb R, \quad \widehat{\deg}(\mathcal T, g) = \sum_{p} \sum_{z
\in \mathcal T(\bar{\mathbb F}_p)} \frac{1}{\#\Aut z} i_p(\mathcal
T, z) \log p + \frac{1}2 \sum_{z \in \mathcal Z(\mathbb C)}
\frac{1}{\#\Aut(z)} g(z).
\end{equation}
Here $i_p(\mathcal T) =\operatorname{Length}( \hat{\O}_{\mathcal T,
z})$ and $\hat{\O}_{\mathcal T, z}$ is the strictly local henselian
ring of $\mathcal T$ at $z$. This way, we obtain a bilinear
map---the Faltings height function
\begin{equation}
h: \, \hCH_\mathbb Q \times Z^2(\mathcal M)_\mathbb Q \rightarrow
\mathbb R, \quad (\hat{\mathcal T}, \mathcal Z) \mapsto
h_{\hat{\mathcal T}}(\mathcal Z) =\widehat{\deg}(j^*\hat{\mathcal
T}),
\end{equation}
which is given by
\begin{equation} \label{newneweq7.4}
h_{\hat{\mathcal T}}(\mathcal Z) = \mathcal Z.\mathcal T +\
\frac{1}2 \sum_{z \in \mathcal Z(\mathbb C)}\frac{1}{\#\Aut(z)} g(z)
\end{equation}
when $\mathcal Z$ and $\mathcal T$ intersect properly.

 Finally, if $\hat{\mathcal L}=(\mathcal L, \|\, \| )$ is a metrized line bundle on
 $\tilde{\mathcal M}$ with a pre-log growth metric along the
 boundary in the sense of \cite[Section 1]{BBK}, let $s$ be a
 rational section of $\mathcal L$, and $\widehat{\div} s=(\div s, -\log \|s\|^2)
 \in \hCH$ is independent of the choice of $s$, and is denoted by $\hat{c}_1(\hat{\mathcal L})$.
 Actually, it only depends on the equivalence class of
 $\hat{\mathcal L}$. We define the Faltings height of $\mathcal Z$
 with respect to $\hat{\mathcal L}$ by
 \begin{equation}
  h_{\hat{\mathcal L}}(\mathcal Z) = h_{\widehat{\div} s}(\mathcal
  Z)
  \end{equation}
which  depends only on the equivalence class of $\hat{\mathcal L}$.

  Let  $\tilde{\mathcal T}_m$ be the closure of the arithmetic Hirzebruch-Zagier divisor $\mathcal T_m$
  in $\tilde{\mathcal M}$.   It is also the flat closure of $\tilde T_m$ where $\tilde T_m$ is the closure of
  the classical Hirzebruch-Zagier divisor $T_m$ in $\tilde{\mathcal M}(\mathbb C)$.
     Bruinier, Burgos-Gil, and K\"uhn defined in
  \cite{BBK} a pre-log-log Green function $G_m$ for  $\tilde T_m$ so that $\hat{\mathcal
  T}_m=(\tilde{\mathcal T}_m, G_m) \in \hCH$.

Let $\omega$  be the Hodge bundle on $\tilde{\CalM}$. Then the
rational sections of  $\omega^k$ can be identified with  meromorphic
Hilbert modular forms for $\SL_2(\O_F)$ of weight $k$. We give it
the following Petersson metric
\begin{align}\label{def:pet}
\|F(z_1,z_2)\|_{\Pet}= |F(z_1,z_2)|\big(16\pi ^2 y_1 y_2\big)^{k/2}
\end{align}
for a Hilbert modular form $F(z)$ of weight $k$. This gives a
metrized Hodge bundle $\hat{\omega}=(\omega, \|\,\|_{\Pet})$. This
metric is shown in \cite[Section 2]{BBK} to have pre-log growth
along the boundary, and so $\hat{c}_1(\hat{\omega}) \in \hCH$.
 It is proved in \cite[Corollary 2.4]{Ya3}
that
\begin{equation} \label{neweq7.4}
h_{\hat{\omega}} (\CM(K)) = \frac{2\# \Cm(K)}{W_K} h_{\Fal}(A)
\end{equation}
for any CM abelian surface $(A, \iota, \lambda) \in \CM(K)(\mathbb
C)$.  The following theorem is proved in \cite{BBK}.

\begin{theo} \label{newtheo7.1} (1)\quad  The generating function
$$
\hat{\phi}(\tau) = -\frac{1}2 \hat{c}_1(\hat{\omega}) + \sum_{m
>0}
  \hat{\mathcal T}_m e(m \tau)
$$
is a modular form of weight $2$, level $D$, and Nebentypus
character $(\frac{D}{})$ with values in $\hCH_{\mathbb Q}$.

(2) \quad Let $\mathcal H\mathcal Z$ be
  the subspace of $\hCH_{\mathbb Q}$ generated by $\hat{\mathcal T}_m$. Then $\mathcal H\mathcal Z$ is a finite dimensional
vector space over $\mathbb Q$.

(3) \quad Let $S$ be the set of primes split in $F$, and let $S_0$
be a finite subset of $S$. Then $\mathcal H\mathcal Z$ is generated
by $\hat{\mathcal T}_q$ with $q \in S-S_0$.
\end{theo}

 {\bf Proof of Theorem \ref{maintheo}}: Now we are ready to prove
 the main result of this paper. We first show that Theorem
 \ref{maintheo} holds for a prime $q$ split in $F$, strengthening
 Theorem \ref{theo1.5}. By Theorem \ref{newtheo7.1}, there are non-zero
 integers $c, c_i$ and   primes $q_i$ ($\ne q$) split in $F$ such that
 $$
 c \hat{\mathcal T}_q = \sum c_i \hat{\mathcal T}_{q_i}.
 $$
 This means that there is a (normalized integral in the sense of \cite[Page 3]{BY}) meromorphic
 function $\Psi$ such that
 $$
 \div \Psi =  c \tilde{\mathcal T}_q - \sum c_i \tilde{\mathcal T}_{q_i}. \quad
 $$
So one has  by (\ref{newneweq7.4}) and  Lemma \ref{newlem2.1}
\begin{align*}
0 &= h_{\widehat{\div}(\Psi)}(\CM(K))
\\
 &=  c \CM(K).\tilde{\mathcal T}_q
-\sum c_i \CM(K).\tilde{\mathcal T}_{q_i}- \frac{2}{W_K} \sum_{z \in
\mathrm{CM}(K)} \log |\Psi(z)|
\\
 &=
 c \CM(K). \mathcal T_q
-\sum c_i \CM(K).\mathcal
  T_{q_i}
     -\frac{2}{W_K} \sum_{z \in \mathrm{CM}(K)} \log |\Psi(z)|.
     \end{align*}
Here we  used the fact that $\CM(K)$ never meets with the boundary
of $\tilde{\mathcal M}$ and thus $\CM(K).\tilde{\mathcal
T}_m=\CM(K).\mathcal T_m$.
 By
 \cite[Theorem 1.1]{BY} (this is the place we need the condition
 that $\tilde D$ is prime), and the fact
 \begin{equation} \label{newneweq7.8}
 W_K =W_{\tilde K} =\begin{cases}
    10 &\ff K=\mathbb Q(\zeta_5),
    \\
     2 &\hbox{otherwise},
     \end{cases}
     \end{equation}
      one has
 $$
\frac{2}{W_K} \sum_{z \in \mathrm{CM} (K)} \log |\Psi(z)|
=\frac{1}{2} c b_q -\frac{1}2 \sum c_i b_{q_i}.
$$
Now applying Theorem \ref{theo1.5}, one has
$$
0 =c(\mathcal T_q.\CM(K) -\frac{1}2 b_q) - \sum c_i (\mathcal
T_{q_i}.\CM(K) -\frac{1}2 b_{q_i})
  =c c_q \log q -\sum c_i c_{q_i} \log q_i
  $$
for some rational numbers $c_q, c_i \in \mathbb Q$. Since $\log q$
and $\log q_i$ are $\mathbb Q$-linearly independent, we have $c_q
=c_{q_i}=0$, and thus
\begin{equation} \label{eq7.3}
\mathcal T_q.\CM(K) =\frac{1}2 b_{q}.
\end{equation}

Now we turn to the general case. Using again Theorem
\ref{newtheo7.1}, there are non-zero integers $c$ and $c_i$ and
primes $q_i$ split in $F$ such that
$$
c  \hat{\mathcal T}_m=\sum c_i \hat{\mathcal T}_{q_i}.
$$
So there is a (normalized integral) Hilbert  meromorphic function
$\Psi$ such that
$$
{\div}(\Psi) =c\tilde{\mathcal T}_m-\sum c_i \tilde{\mathcal
T}_{q_i}.
$$
So one has by (\ref{newneweq7.4}),  (\ref{eq7.3}) and \cite[Theorem
1.1]{BY}
\begin{align*}
0 &=h_{\widehat{\div}(\Psi)} (\CM(K))
\\
   &=c \CM(K).{\mathcal T}_m -\sum c_i\CM(K). \mathcal T_{q_i}
     -\frac{2}{W_K}\sum_{z \in \mathrm{CM}(K)} \log |\Psi(z)|
     \\
     &=c \CM(K).\mathcal T_m -\frac{1}2 c b_m.
     \end{align*}
Therefore $\mathcal T_m.\CM(K) =\frac{1}2 b_m$. This proves Theorem
\ref{maintheo}.

{\bf Proof of Theorem \ref{Colmez}}: By \cite[Theorems 4.15,
5.7]{BBK}, there is a normalized integral meromorphic  Hilbert
modular form $\Psi$ of weight $c(0) >0$ such that
$$
\div \Psi = \sum_{m>0} c_m \tilde{\mathcal T}_m.
$$
Now  the same argument as in the proof of \cite[Theorem 1.5]{m=1}
gives
\begin{equation} \label{neweq7.7}
h_{\hat{\omega}}(\CM(K)) = \frac{\#\Cm(K)}{W_K} \beta(K/F).
\end{equation}
Combining this with (\ref{neweq7.4}), one proves the theorem.

To state Theorem \ref{theo1.4} more precisely and prove it, we need
some preparation. Let
\begin{equation} \label{neweq7.8}
E_2^+(\tau) = 1+ \sum_{m>0} C(m, 0)e(n\tau), \quad C(m, 0) =
\frac{2\sum_{d|m} d}{L(-1, (\frac{D}{}))}
\end{equation}
be the Eisenstein series of weight $2$, level $D$, and Nebentypus
character $(\frac{D}{})$ given in \cite[Corollary 2.3]{BY}.

Let $\chi_{\tilde K/\tilde F}$ be the quadratic Hecke character of
$\tilde F$ associated to $\tilde K/\tilde F$, and let $I(s,
\chi_{\tilde K/\tilde F})$ be the induced representation of
$\SL_2(\A_{\tilde F})$. In \cite[Section 6]{BY}, we choose a
specific section $\Phi \in I(s, \chi_{\tilde K/\tilde F})$ and
constructed an (incoherent) Eisenstein series of weight $1$
$$
E^*(\tau_1, \tau_2, s, \Phi) = (v_1 v_2)^{-\frac{1}2}
E(g_{\tau_1} g_{\tau_2}, s, \Phi) \Lambda(s+1,
 \chi_{\tilde K/\tilde F}).
$$
Here $ \tau_j = u_j+  i v_j\in \mathbb H$. 
 The Eisenstein series
is automatically zero at $s=0$. So its diagonal restriction of
$\mathbb H$ is a modular  form of weight $2$, level $D$, Nebentypus
character $(\frac{D}{})$ which is zero at $s=0$.  Let
$$
\tilde f(\tau)=\frac{1}{ \sqrt D} E^{*, \prime}(\tau, \tau, 0,
\Phi)|_2W_D
$$
be the modular form   defined in \cite[(7.2)]{BY}) (with $K$ in
\cite[Sections 7 and 8]{BY} replaced by $\tilde K$). Here $W_D=\kzxz
{0} {-1} {D} {0}$. Finally let $f$ be the holomorphic projection of
$\tilde f$. According to \cite[Theorem 8.1]{BY}, one has the Fourier
expansion
\begin{equation}
f(\tau) = -4\sum_{m>0} (b_m +c_m + d_m) e(m\tau)
\end{equation}
where $b_m$ is the number in  Conjecture \ref{conj},
\begin{equation}
d_m= \frac{1}2 C(m, 0)\Lambda(0, \chi_{\tilde K/\tilde F})
\beta(\tilde K/\tilde F)
\end{equation}
and $c_m$ is some complicated constant defined in \cite[Theorem
8.1]{BY}.  Notice that the Green function $G_m$ in $\hat{\mathcal
T}_m$ is also the Green function used in \cite{BY}. So
\cite[(9.3)]{BY} gives ($\CM(K)$ in \cite{BY} is our $\Cm(K)$)
\begin{equation}
c_m =\frac{4}{W_{\tilde K}}G_m(\Cm(K)) =\frac{4}{W_K} G_m(\Cm(K)).
\end{equation}
As explained in the proof of \cite[Theorem 1.5]{m=1}, one has
$$
\Lambda(s, \chi_{\tilde K/\tilde F}) = \Lambda(s, \chi_{K/F}).
$$
So $\beta(\tilde K/\tilde F)=\beta(K/F)$. One has also by
\cite[(9.2)]{BY} and (\ref{newneweq7.8})
\begin{equation}
\Lambda(0, \chi_{\tilde K/\tilde F}) = \frac{2 \# \Cm(K)}{W_K}.
\end{equation}
So (\ref{neweq7.7}) implies
\begin{equation}
d_m = h_{\hat{\omega}}(\CM(K))C(m ,0).
\end{equation}
So we have
\begin{equation} \label{neweq7.14}
f(\tau) =-4 \sum_{m>0} (b_m +\frac{4}{W_K} G_m(\Cm(K))) e(m\tau)
          -4h_{\hat{\omega}}(\CM(K)) \sum_{m>0} C(m, 0) e(m\tau).
          \end{equation}

 Now we can restate Theorem \ref{theo1.4}
 more precisely:

\begin{theo} \label{theo7.1} Let the notation be as above. Assuming (\ref{eqOK}) and
that $\tilde D=\Delta \Delta' \equiv 1 \mod 4$ is a prime. Then
$$
h_{\hat\phi}(\CM(K))
 +\frac{1}2 h_{\hat{\omega}}(\CM(K)) E_2^+(\tau)  =-\frac{1}{8} f(\tau).
 $$
 \end{theo}
\begin{proof} By  Theorem \ref{maintheo}, (\ref{newneweq7.4}),  and (\ref{neweq7.7}), we have
\begin{align*}
h_{\hat\phi}(\CM(K))
 &=-\frac{1}2 h_{\hat{\omega}}(\CM(K)) + \sum_{m>0}
 h_{\hat{\mathcal T}}(\CM(K)) e(m\tau)
 \\
 &=-\frac{1}2 h_{\hat{\omega}}(\CM(K))
  + \sum_{m >0} (\CM(K).\mathcal T_m + \frac{2}{W_K}
  G_m(\hbox{CM}(K))) e(m\tau)
  \\
   &=-\frac{1}2 h_{\hat{\omega}}(\CM(K)) + \frac{1}2\sum_{m>0}
  (
   b_m + \frac{4}{W_K}
  G_m(\hbox{CM}(K))) e(m\tau).
  \end{align*}
Combining this with (\ref{neweq7.8}) and (\ref{neweq7.14}), one
proves the theorem.
\end{proof}

\section{Siegel modular variety of genus $2$ and Lauter's conjecture}
\label{sect8}

Following \cite{CF}, let $\mathcal A_2$ be the moduli stack over
$\mathbb Z$ representing the principally
 polarized abelian surfaces $(A, \lambda)$. Then
  $\mathcal A_2(\mathbb C)= \Sp_2(\mathbb Z)\backslash \mathbb H_2
  $  is  the Siegel modular surface of genus $2$.   Here
  $\mathbb H_2 = \{ Z\in \operatorname{Mat}_2(\C);\; Z={}^tZ, \,
\Im(Z)>0\}$ is the Siegel upper half plane of genus two. Let
$\epsilon$ be a fixed fundamental unit if $F=\mathbb Q(\sqrt D)$
with $\epsilon >0$ and $\epsilon' <0$. Then
 \begin{equation}
 \phi_D: \mathcal M \rightarrow \mathcal A_2, \quad (A, \iota,
 \lambda)  \mapsto (A, \lambda(\frac{\epsilon}{\sqrt D}))
 \end{equation}
 is a natural map from $\mathcal M$ to $\mathcal A_2$, which is
 proper and generically $2$ to $1$. For an integer $m \ge 1$, let
 $G_m$ be the Humbert surface in $\mathcal A_2(\mathbb Q)$
 \cite[Chapter IX]{Ge}, defined  as follows (over $\mathbb C$). Let $L=\mathbb
 Z^5$ be with the quadratic form
 $$
 Q(a, b, c, d, e) = b^2 - 4 ac-4 de.
 $$
  We remark that there is an isomorphism between $\Sp_2(\mathbb Q)/\{\pm 1\}$
 and  $\operatorname{SO}(L\otimes \mathbb Q)$.  For $x \in L$ with $Q(x) >0$, we define
$$
H_x=\{ \tau =\kzxz {\tau_1}  {\tau_2} {\tau_2} {\tau_3} \in
\mathbb H_2:\, a \tau_1 + b \tau_2 + c\tau_3 + d (\tau_2^2-\tau_1
\tau_3) + e=0\}.
$$
Then $H_x$ is a copy of $\mathbb H^2$ embedded into $\mathbb H_2$.
The Humbert surface $G_m$ is then defined by
\begin{equation}
G_m = \Sp_2(\mathbb Z) \backslash \{ H_x:\, x \in L, Q(x)
=m\}.
\end{equation}

  Let $\mathcal G_m$ be
 the flat closure of $G_m$ in $\mathcal A_2$. Then $(\phi_D)_*
 \mathcal M =2 \mathcal G_D$, and
 \begin{equation} \label{neweq8.2}
 \phi_D^* \mathcal G_m = \sum_{n>0, \frac{Dm-n^2}4 \in  \mathbb Z_{>0}} \mathcal T_{\frac{Dm-n^2}4}
 \end{equation}
 when $m D$ is a not a square. Indeed, it is known \cite[Theorem 3.3.5]{Fr}, \cite[Proposition IX 2.8]{Ge} that
 $$
 \phi_D^*  G_m = \sum_{n>0, \frac{Dm-n^2}4 \in  \mathbb Z_{>0}} T_{\frac{Dm-n^2}4}.
 $$
So their flat closures in $\mathcal M$ are equal too, which is
(\ref{neweq8.2}).

 Let $K$ be a quartic CM number field with real quadratic subfield $F$, and let
 $\CM_S(K)$ be the moduli stack over $\mathbb Z$ representing the
  moduli problem which assigns a scheme $S$ the set of triples
  $(A, \iota, \lambda)$ where $(A, \lambda) \in \mathcal A_2(S)$
  and $\iota$ is an $\O_K$-action on $A$ such that the Rosati
  involution associated to $\lambda$ gives the complex
  conjugation on $K$. Notice that the map
  $$
  \CM(K) \rightarrow \CM_S(K), \quad (A, \iota, \lambda)  \mapsto
  (A, \iota, \lambda(\frac{\epsilon}{\sqrt D}))
  $$
  is an isomorphism of stacks.
  We also denote $\CM_S(K)$ for the direct
  image of $\CM_S(K)$  in $\mathcal A_2$ under the forgetful map
  (forgetting the $\O_K$ action). Then the above isomorphism
  implies that $(\phi_D)_*(\CM(K)) = \CM_S(K)$. Now the proof of
  Theorem \ref{siegel} is easy.

  {\bf Proof of Theorem \ref{siegel}}: By  the projection formula, Theorem \ref{maintheo},  and
  remarks above, one has
  \begin{align*}
  \CM_S(K). \mathcal G_m
   &=(\phi_D)_*(\CM(K)).\mathcal G_m
   \\
    &= \CM(K). \phi_D^*(\mathcal G_m)
    \\
     &= \sum_{n>0, \frac{Dm-n^2}{4} \in \mathbb Z_{>0}}
     \CM(K).\mathcal T_{\frac{Dm-n^2}4}
     \\
      &=\frac{1}2 \sum_{n>0, \frac{Dm-n^2}{4} \in \mathbb
      Z_{>0}}b_{\frac{Dm-n^2}4}
      \end{align*}
as claimed.

To describe and prove Lauter's conjecture on Igusa invariants, we
need more notations. Let
\begin{equation}
\theta_{a, b}(\tau, z) =\sum_{n \in \mathbb Z^2} e^{\pi i
{}^t(n+\frac{1}2 a) \tau (n+ \frac{1}2 a)+ 2 {}^t(n+\frac{1}2 a)(z+
\frac{1}2 b)}
\end{equation}
be the theta functions  on $\mathbb H_2 \times \mathbb C^2$ with
characters $a, b \in (\mathbb Z/2)^2$. It is zero at $z=0$  unless
${}^ta b \equiv 0 \mod 2$. In such a case, we call $\theta_{a,
b}(\tau, 0)$ an even theta constants. There are exactly ten of them,
we renumbering them as $\theta_i$, $1 \le i \le 10$. They  are
Siegel modular forms of weight $1/2$ and some level.
$$
h_{10}=\prod_{i} \theta_i^2
$$
is a cusp form of weight $10$ and  level $1$ and is the famous Igusa
cusp form $\chi_{10}$. Igusa also defines in \cite{Ig} three other
Siegel modular forms $h_4=\sum_i \theta_i^8$, $h_{12}$, and $h_{16}$
for $\hbox{Sp}_2(\mathbb Z)$ of weight $4$, $12$, and $16$
respectively as polynomials of these even theta constants. We refer
to \cite{Wen} for the precise definition of $h_{12}$ and $h_{16}$
since they are complicated and not essential to us. The so-called
$3$ Igusa invariants are defined as (\cite[Section 5]{Wen}
\begin{equation}
j_1 = \frac{h_{12}^5}{h_{10}^6}, \quad j_2 = \frac{h_4
h_{12}^3}{h_{10}^4}, \quad j_3 = \frac{h_{16}h_{12}^2}{h_{10}^4}.
\end{equation}

It is known that $h_i$ have integral Fourier coefficients.  Since
four of ten theta constants have constant term $1$ and the other
six are multiples  of $2$, one can  check (\cite{GN})
$$
h_{10} = 2^{12} \Psi_{1, S}
$$
where $\Psi_{1, S}$ is an integral Siegel modular form for
$\Sp_2(\mathbb Z)$ with  constant term $1$ and $\div \Psi_{1, S} =2
\mathcal G_1$.  One can also check
$$
h_4=2^4 \tilde h_4, \quad h_{12} =2^{15} \tilde h_{12}, \quad
h_{16}=2^{15} \tilde h_{16}
$$
with $\tilde h_4, \tilde h_{12}$, and $\tilde h_{16}$ still having
integral coefficients. So
\begin{equation} \label{neweq8.5}
j_1 = 2^3 \frac{\tilde h_{12}^5}{\Psi_{1, S}^{6}}, \quad j_2 = 2
\frac{\tilde h_4 \tilde h_{12}}{\Psi_{1, S}^4}, \quad j_3 =2^{-3}
\frac{\tilde h_{12} \tilde h_{16}}{\Psi_{1, S}^4}.
\end{equation}
We renormalize
\begin{equation}
j_1=2^3 B_1 j_1',  \quad j_2 =2 B_2 j_2', \quad j_3 =2^{-3} B_3 j_3'
\end{equation}
for some positive integers $B_i$ so that $j_i'$ can be written as
$$
j_i'=\frac{f_i}{\Psi_{1, S}^{n_i}}$$ with $n_1=6, n_2=n_3=4$
such  that $f_i$ are integral Siegel modular forms whose Fourier
coefficients have greatest common divisor $1$.
.

 Let $K$ be a quartic non-biquadratic CM number field with real
 quadratic subfield $F=\mathbb Q(\sqrt D)$. For a CM type $\Phi$ of $K$,
 let $\Cm_S(K, \Phi)$ be the formal
 sum of principally polarized abelian surfaces over $\mathbb C$
 of CM type $(\O_K, \Phi)$ (up to isomorphism). It is the image of
 $\Cm(K, \Phi)$ under $\phi_D$. So $\Cm_S(K)=\Cm_S(K, \Phi_1)
 +\Cm(K, \Phi_2)$ is defined over $\mathbb Q$ and
 $$
 \CM_S(K)(\mathbb C) = 2\Cm_S(K).
 $$
Here $\Phi_1$ and $\Phi_2$ are two CM types of $K$ such that
$\Phi_i$ and $\rho\Phi_i$ give all CM types of $K$ ($\rho$ is the
complex conjugation). By the theory of complex multiplication
\cite[Main Theorem 1, page 112]{Sh},
$$
j_i'(\hbox{CM}_S(K)):=\prod_{z \in \Cm_S(K)} j_i'(z)
$$
is a power of $ \norm(j_i'(z))$ for any CM point $z \in \Cm_S(K)$.
So Theorem \ref{Lauter} is a consequence of the following theorem.

{\bf Proof of Theorem \ref{Lauter}}:
 We prove the theorem for $j_1'$. The proof for $j_2'$
and $j_3'$ is  the same. We first prove $ A_1 \norm(j_1'(\tau))
\in \mathbb Z$. By the theory of complex multiplication \cite[Main
Theorem 1, page 112]{Sh},
$$
j_i'(\hbox{CM}_S(K)):=\prod_{\tau \in \Cm_S(K)} j_i'(\tau)
$$
is a power of $ \norm(j_i'(\tau))$ for any CM point $\tau \in
\Cm_S(K)$. Since $\CL_0(K) \cong \CL_0(\tilde K)$ in our case by
\cite[Lemma 5.3]{BY}, we have actually
$j_i'(\hbox{CM}_S(K))=\norm(j_1'(\tau))$.

Notice that
$$
\div  j_1' = \div f_1  -12 \mathcal G_1.
$$
If $\CM(K)$ and $\div f_1$ intersect improperly,  they have a
common point over $\mathbb C$ (since both are horizontal). So
$f_1(\hbox{CM}(K) =0$ and  $j_1'(\hbox{CM}_S(K))=0$,  there is
nothing to prove. So we may assume $\CM(K)$ and $\div f_1$
intersect properly. Since both are effective cycles, one has $$
\CM(K).\div f_1 =a \log C$$ for some positive integer $C
>0$ and a rational number $a >0$. Now
\begin{align*}
0&=h_{\widehat{\div} j_1'}(\CM(K))
\\
 &=\CM(K).\div f_1 -12 \CM(K).\mathcal
 G_1 - \frac{2}{W_K} \log |j_1'(\hbox{CM}_S(K))|
 \\
  &=\CM(K).\div f_1-6 \sum_{0 < n <\sqrt D, odd} b_{\frac{D-n^2}4}
    -\frac{2}{W_K}\log|j_1'(\hbox{CM}_S(K)|.
    \end{align*}
Write $\norm(j_1'(\tau)) =M_1/N_1$ with $(M_1, N_1)=1$. Then
$$
\log |M_1| -\log N_1= \log|j_1'(\hbox{CM}_S(K)| = \frac{a W_K}2
\log C -3 W_K \sum_{0 < n <\sqrt D, odd} b_{\frac{D-n^2}4},
$$
and so
$$\log N_1 =3 W_K
\sum_{0 < n <\sqrt D, odd} b_{\frac{D-n^2}4} +\log |M_1| -\frac{a
W_k}2 \log C= \log A_1 +\log |M_1| -\frac{a W_k}2 \log C.
$$
So $N_1 C^{\frac{aW_K}2} =A_1 |M_1|$,  and thus $N_1 |A_1$. $A_1
\norm(j_1'(\tau)) \in \mathbb Z$.

We now derive $A_1 H_1(x) \in \mathbb Z$. The $k$-th coefficient
of $H_1(x)$ is
$$
a_k= \sum_{ i_1 \le i_2 \le \cdots \le i_k}  j_1'(\tau_{i_1})
\cdots j_1'(\tau_{i_k})
$$
where $\tau_{j} \in \CM_S(K)$.  Write
$$
j_1'(\tau_j)\O_L =\frac{\mathfrak a_j}{\mathfrak b_j}
$$
uniquely with $\mathfrak a_j, \mathfrak b_j$ being integral ideals
of $\O_L$, where $L$ is a Galois extension of $\mathbb Q$
containing all $j_1'(\tau_j)$. Then  $N_1\mathbb Z =\prod
\mathfrak b_j$.  So
$$
a_k \mathbb Z = \mathfrak c/N_1
$$
where
$$
\mathfrak c=\sum \prod_{l=1}^k \mathfrak a_{i_l} \prod_{j \ne i_l}
\mathfrak b_j
$$
is an integral ideal of $L$. So $\mathfrak c= c\mathbb Z$ for some
integer $c$, and thus $a_k = \pm c/N_1$. That is $A a_k \in
\mathbb Z$. This proves Theorem \ref{Lauter}


\begin{thebibliography}{EMOT}\label{secref}

\bibitem[An]{An} {\em G. Anderson} Logarithmic derivatives of Dirichlet $L$-functions and the periods of abelian varieties.
 Compositio Math.  {\bf 45}  (1982),  315--332.

\bibitem[BBK]{BBK} {\em J. Bruinier, J. Burgos Gil, and U. K\"uhn},
 Borcherds products and arithmetic intersection theory on Hilbert
 modular surfaces, Duke Math. J. {\bf 139} (2007), 1-88.

\bibitem[BKK]{BKK} {\em J. Burgos Gil, J. Kramer,  and U. K\"uhm
},  Cohomological arithmetic Chow rings, J. Inst. Math. Jussieu
{\bf 6}(2007), 1-172.

\bibitem[BY]{BY} {\em J. H. Bruinier and T. H. Yang}, CM values of
Hilbert modular functions, Invent. Math. {\bf 163} (2006), 229--288.


\bibitem[CF]{CF} {\em G. Faltings and C. L. Chai},  Degeneration
of abelian varieties,  Springer-Verlag, 1990.

\bibitem[CS]{CS} {\em C. Chowla and A. Selberg}, On Epstein's zeta-function,
    J. Reine Angew. Math.  {\bf 227}(1967), 86--110.

\bibitem[CL]{CL} {\em H. Cohn and K. Lauter} Generating genus two
curves with complext multiplication, Microsoft Internal Techincal
Report, January, 2001.



 \bibitem[Co]{Col} {\em P. Colmez}, P\'eriods des vari\'et\'es
 ab\'eliennes \`a multiplication complex, Ann. Math., 138(1993),
 625-683.


\bibitem[Co2]{Col1} {\em P. Colmez},  Périodes de variétés abéliennes à multiplication
complexe et dérivées de fonctions $L$ d'Artin en $s=0$.  C. R. Acad.
Sci. Paris Sér. I Math. 309 (1989),  139--142.
\bibitem[Fa]{Faltings} {\em G. Faltings}, Finiteness theorems for abelian varieties over number fields.
Translated from the German original [Invent. Math. 73 (1983), no. 3,
349--366; ibid. 75 (1984), no. 2, 381; MR 85g:11026ab] by Edward
Shipz. Arithmetic geometry (Storrs, Conn., 1984),  9-27, Springer,
New York, 1986.
\bibitem[Fr]{Fr} {\em H.-G. Franke}, Kurven in Hilbertschen Modulfl\"achen und Humbertsche Fl\"achen im Siegelraum,
Bonner Math. Schriften {\bf 104} (1978).



\bibitem[Ge]{Ge} {\em G. van der Geer}, Hilbert Modular Surfaces, Springer-Verlag (1988).

\bibitem[Gi]{Gi} {\em H. Gillet},  Intersection theory on algebraic stacks and
Q-varieties. In Proceedings of the Luminy conference on algebraic
K-theory (Luminy, 1983), volume 34, pages 193–-240, 1984.

\bibitem[Go]{Go} {\em E. Goren}, Lectures on Hilbert modular
varieties and modular forms, CRM monograph series 14, 2001.



\bibitem[GN]{GN} {\em V. A. Gritsentko and V.V. Nikulin},
Siegel automorphic form corrections of some lorentzian kac–moody lie
algebras, Amer.  J.  Math. {\bf  119} (1997), 181–224.



\bibitem[Gr]{Gr} {\em B. Gross} On the periods of abelian integrals and a formula of Chowla and Selberg.
With an appendix by David E. Rohrlich.  Invent. Math.  {\bf 45}
(1978), 193-211.

\bibitem[GK]{GK}  {\em B. Gross and K. Keating}, On the intersection
of modular correspondences, Invent. Math. {\bf 112}(1993), 225--245.

\bibitem[GZ1]{GZ} {\em B. Gross and D. Zagier}, Heegner points and derivatives of $L$-series.  Invent. Math. {\bf 84}  (1986),  225--320.

\bibitem[GZ2]{GZ2} {\em B. Gross and D. Zagier}, On singular moduli.  J. Reine Angew. Math.  {\bf 355 } (1985), 191--220.


\bibitem[HZ]{HZ} {\em F. Hirzebruch and D. Zagier}, Intersection Numbers of Curves on Hilbert Modular
Surfaces and Modular Forms of Nebentypus, Invent. Math. {\bf 36}
(1976), 57--113.


\bibitem[Ho]{Ho} {\em B. Howard},  Intersection theory on Shimura
surfaces,  preprint 2008, pp45.



\bibitem[Ig1]{Ig} {\em J.-I. Igusa}, Arithmetic Variety of Moduli for Genus Two.
Ann. Math. {\bf 72 }(1960), 612–-649

\bibitem[Ig2]{Ig2} {\em J.-I. Igusa}, Modular Forms and Projective
Invariants,  American Journal of Mathematics, {\bf 89}(1967),
817--855.

\bibitem[La]{La} {\em K. Lauter},
Primes in the denominators of Igusa Class Polynomials, preprint,
pp3, http://www.arxiv.org/math.NT/0301240/

\bibitem[KRo]{KRo} {\em K. K¨ohler and D. Roessler},
 Afixed point formula of Lefschetz type in Arakelov
geometry. IV. The modular height of C.M. abelian varieties,  J.
Reine Angew. Math. 556 (2003), 127–-148.

\bibitem[KZ]{KZ} {\em M. Kontsevich and D. Zagier},  Periods.
Mathematics unlimited---2001 and beyond, 771--808, Springer, Berlin,
2001.


\bibitem[Ku1]{KuAnnal} {\em S. Kudla}, Central derivatives of Eisenstein series and height pairings.
 Ann. of Math. (2)  {\bf 146 } (1997),  545--646.

\bibitem[Ku3]{KuICM} {\em S. Kudla}, Derivatives of Eisentein series
and arithmetic geometry, Publ. ICM, Vol II (Beijing 2002), 173--183,
Higher Education Press, Beijing, 2002.


\bibitem[Ku2]{KuMSRI} {\em S. Kudla},   Special cycles and derivatives of Eisenstein series, in  Heegner points and Rankin L-series,
243-270, Math. Sci. Res. Inst. Publ., 49, Cambridge Univ. Press,
Cambridge, 2004.

\bibitem[KR1]{KR1} {\em S. Kudla and M. Rapoport},   Arithmetic Hirzebruch Zagier cycles, J. reine angew. Math., {\bf 515 }(1999),
155-244.

\bibitem[KR2]{KR2} {\em S. Kudla and M. Rapoport}, Cycles on Siegel 3-folds and derivatives of Eisenstein series,
 Annales Ecole. Norm. Sup. {\bf 33} (2000), 695-756.

\bibitem[KRY1]{KRY1}  {\em S. Kudla, M. Rapoport, and T.H. Yang},
 Derivatives of Eisenstein Series and Faltings heights,   Comp. Math., {\bf 140 }(2004),   887-951.

\bibitem[KRY2]{KRY2}  {\em S. Kudla, M. Rapoport, and T.H. Yang},  Modular forms and special cycles on
Shimura curves,  Annals of Math. Studies series, vol 161, Princeton
Univ. Publ., 2006.

\bibitem[KRY3]{KRY3}  {\em S. Kudla, M. Rapoport, and T.H. Yang},
On the  derivative of Eisenstein Series of weight one, IMRN (1999),
347-385.


\bibitem[MR]{MR} {\em  V. Maillot and D.  Roessler},  On the periods of motives with
complex multiplication and a conjecture of Gross-Deligne,   Ann. of
Math. (2)  {\bf 160}  (2004),   727--754.

\bibitem[Pa]{Pa} {\em G. Pappas}, Arithmetic models for Hilbert modular varieties,
Compos. Math. 98 (1995), 43-76.


\bibitem[Ru]{Ru} {\em B. Runge}, Endomorphism rings of abelian surfaces and projective models of their moduli spaces,
Tohoku Math. J. {\bf 51} (1999), 283--303.

\bibitem[Se]{Se} {\em J.-P. Serre}, A course in Arithmetic, GTM
{\bf 7}, Springer-Verlag, New York, 1973.

\bibitem[Sh]{Sh} {\em G. Shimura}, Abelian varieties with complex
multiplication and modular functions,  Princeton Math. Series \vol
46, Princeton Univ. Press, 1997.

\bibitem[Vi]{Vi} {\em  A. Vistoli},  Intersection theory on algebraic stacks and on
their moduli spaces,  Invent. Math., 97(1989), 613–-670.

\bibitem[Vo]{Vo} {\em I. Vollaard},  On the Hilbert-Blumenthal moduli
problem,
  J. Inst. Math. Jussieu  4  (2005),
653--683.

\bibitem[We1]{We1} {\em T. Wedhorn}, The genus of the endomorphisms
of a supersingular elliptic curve, Chapter 5 in ARGOS seminar on
Intersections of Modular Correspondences, p. 37-58, to appear in
Asterisque.

\bibitem[We2]{We2} {\em T. Wedhorn}, Caculation of representation densities, Chapter 15
in ARGOS seminar on Intersections of Modular Correspondences, p.
185-196, to appear in Asterisque.

\bibitem[Wen]{Wen} {\em A. Weng}  Constructing hyperelliptic curves
of genus two suitable for cryptography, Math. Comp., {\bf 72}(2002),
435-458.


\bibitem[Ya1]{YaDensity1} {\em T.H. Yang},
  An explicit formula for local densities of quadratic
  forms, J. Number Theory  {\bf 72}(1998), 309-356.

\bibitem[Ya2]{YaDensity2} {\em T.H. Yang}, Local densities of $2$-adic quadratic
forms,  J. Number Theory  {\bf 108}(2004), 287-345.

\bibitem[Ya3]{Ya3} {\em T. H. Yang},Chowla-Selberg Formula and Colmez's Conjecture,
 Accepted to appear in Canada J. Math., pp17.

\bibitem[Ya4]{m=1}  {\em T. H. Yang}, An arithmetic intersection formula on Hilbert modular surfaces,
 accepted to appear in Amer. J. Math., pp30.

\bibitem[Ya5]{YaICCM3} {\em T. H. Yang}, Hilbert modular functions and their CM values, Proc.
 of  the 3rd ICCM, AMS/IP Studies in Adv. Math. 42(2008), 135-154.


\bibitem[Yo]{Yo} {\em  H. Yoshida},  Absolute CM-periods. Mathematical Surveys and
Monographs, {\bf  106}, American Mathematical Society, Providence,
RI, 2003.



\bibitem[Yu]{Yu} {\em C.F. Yu},  The isomorphism classes of abelian
varieties of  CM types, Jour. pure and Appl. Algebra, {\bf
187}(2004), 305-319.

\bibitem[Zh1]{Zh1} {\em S.W. Zhang},
Heights of Heegner cycles and derivatives of L-series, Invent. Math.
{\bf 130} (1997), no. 1, 99--152.

\bibitem[Zh2]{Zh2} {\em S. W. Zhang},
Heights of Heegner points on Shimura curves,
 Ann. of Math. (2) {\bf 153} (2001), no. 1, 27--147.

\bibitem[Zh3]{Zh3} {\em S. W. Zhang},
Gross-Zagier formula for GL(2), Asian. J. Math. 5 (2001), no. 2,
183--290; II, Heegner points and Rankin L-series, 191--214, Math.
Sci. Res. Inst. Publ., 49, Cambridge Univ. Press, Cambridge, 2004.



\end{thebibliography}
\end{document}